\documentclass[11pt, reqno, english]{amsart}

\usepackage[utf8]{inputenc}
\usepackage[T1]{fontenc}
\usepackage{amsmath,amsthm}
\usepackage{amsfonts,amssymb}
\usepackage{url}
\usepackage{mathtools}  
\usepackage[colorlinks=true,urlcolor=blue,linkcolor=blue,citecolor=blue]{hyperref}
\usepackage{enumerate, paralist}
\usepackage{placeins}
\usepackage{tikz,tikz-cd}
\usetikzlibrary{calc,arrows}
\usetikzlibrary{math} 
\usetikzlibrary{arrows,decorations,backgrounds,calc,math}
\tikzstyle{map}=[->,semithick]
\tikzstyle{arc}=[bend left,->,semithick]
\tikzstyle{rinclusion}=[right hook->,semithick]
\tikzstyle{linclusion}=[left hook->,semithick]

\usepackage[left=1in,right=1in,top=1in,bottom=1in]{geometry}
\usepackage{xcolor}
\usepackage{hyperref}
\usepackage{float}
\usepackage{subfigure}

\usepackage{thm-restate}
\newtheoremstyle{myremark} 
{7pt}                    
{7pt}                    
{}  	                 
{}                           
{\bf}       	         
{.}                          
{.5em}                       
{}  

\theoremstyle{plain}
\newtheorem{theorem-main}{Theorem}
\newtheorem*{theorem-main1}{Theorem~\ref{thm:main1}}
\newtheorem{lemma}{Lemma}[section]
\newtheorem{theorem}[lemma]{Theorem}

\theoremstyle{definition}

\newtheorem*{definition-cnk}{Definition~\ref{def:cnk}}
\newtheorem{question}{Question}
\newtheorem*{question-motivating}{Motivating Question}


\newcounter{parentnumber}

\theoremstyle{myremark}
\newtheorem{remark}[lemma]{Remark}

\newcommand{\C}{\mathbb{C}}

\newcommand{\Q}{\mathbb{Q}}
\newcommand{\R}{\mathbb{R}}
\newcommand{\Z}{\mathbb{Z}}
\newcommand{\F}{\mathbb{F}}

\newcommand{\CP}{\ensuremath{\mathbb{C}\mathrm{P}}}

\newcommand{\diam}{\mathrm{diam}}

\newcommand{\tor}{\mathrm{Tor}}

\newcommand{\id}{\mathrm{id}}

\newcommand{\so}{\mathrm{SO}}

\newcommand{\im}{\mathrm{im}}

\newcommand{\supp}{\mathrm{supp}}

\newcommand{\vr}[2]{\mathrm{VR}(#1;#2)}

\newcommand{\vrm}[2]{\mathrm{VR}^\mathrm{m}(#1;#2)}

\newcommand{\cech}[2]{\mathrm{\check{C}}(#1;#2)}
\newcommand{\cechm}[2]{\mathrm{\check{C}}^\mathrm{m}(#1;#2)}

\newcommand{\complex}[1]{\mathbb{C}_{#1}}

\DeclareMathOperator{\relint}{\mathrm{relint}}

\usepackage{accents}
\usepackage[normalem]{ulem}
\usepackage{wasysym}
\usepackage{charter}
\usepackage{euler}
\usepackage[T1]{fontenc}
\usepackage{thmtools}
\usepackage{thm-restate}

\def\pt{\mathrm{pt}}

\begin{document}

\title{Persistent equivariant cohomology}

\author{Henry Adams}
\address[HA]{Department of Mathematics, University of Florida, Gainesville, FL 32611, USA}
\email{henry.adams@ufl.edu}

\author{Evgeniya Lagoda}
\address[EL]{Institut f\"ur Mathematik, Freie Universit\"at Berlin, 14195 Berlin, Germany}
\email{e.lagoda@fu-berlin.de}

\author{Michael Moy}
\address[MM]{Department of Mathematics, Colorado State University, Fort Collins, CO 80523, USA}
\email{michael.moy@colostate.edu}

\author{Nikola Sadovek}
\address[NS]{Institut f\"ur Mathematik, Freie Universit\"at Berlin, 14195 Berlin, Germany}
\email{nikola.sadovek@fu-berlin.de}

\author{Aditya De Saha}
\address[AdS]{Department of Mathematics, University of Florida, Gainesville, FL 32611, USA}
\email{a.desaha@ufl.edu}

\begin{abstract}
This article has two goals.
First, we hope to give an accessible introduction to persistent equivariant cohomology.
Given a topological group $G$ acting on a filtered space, persistent Borel equivariant cohomology measures not only the shape of the filtration, but also attributes of the group action on the filtration, including in particular its fixed points.
Second, we give an explicit description of the persistent equivariant cohomology of the circle action on the Vietoris--Rips metric thickenings of the circle, using the Serre spectral sequence and the Gysin homomorphism.
Indeed, if $\frac{2\pi k}{2k+1} \le r < \frac{2\pi(k+1)}{2k+3}$, then 
$H^*_{S^1}(\vrm{S^1}{r})\cong
\Z[u]/(1\cdot3\cdot5\cdot\ldots \cdot (2k+1)\, u^{k+1})$
where $\deg(u)=2$.
\end{abstract}

\maketitle

\setcounter{tocdepth}{1}
\tableofcontents

\section{Introduction}
In this paper, we give an accessible introduction to persistent Borel equivariant cohomology, we explore the relationship between persistent equivariant cohomology and related topics, and we compute the persistent equivariant cohomology of an important initial example: the circle action on the Vietoris--Rips metric thickenings of the circle.
The question ``\emph{What is the persistent equivariant cohomology of the Vietoris--Rips filtration of the circle?}'' arose in applied topology, but its answer requires tools from equivariant homotopy theory.
We hope these ideas will find connections to future research projects at the intersection of applied and equivariant topology.

The setup for persistent Borel equivariant cohomology is as follows.
Let $G$ be a topological group.
We say that a space $Y$ is a \emph{$G$-space} if $Y$ is equipped with a group action $G \times Y \to Y$.
We say that 
\[Y_1 \hookrightarrow Y_2 \hookrightarrow \ldots \hookrightarrow Y_{k-1} \hookrightarrow Y_k\]
is a \emph{filtration of $G$-spaces} if each $Y_i$ is a $G$-space, and if each inclusion $h_i \colon Y_i \hookrightarrow Y_{i+1}$ is $G$-equivariant, i.e.\ if $h_i(g\cdot y)=g\cdot h_i(y)$ for all $y\in Y_i$ and $g \in G$.
If we apply the equivariant cohomology functor $H^*_G(\cdot)$ to a filtered $G$-space, then by contravariance we obtain the sequence
\[H^*_G(Y_k) \to H^*_G(Y_{k-1}) \to \ldots \to H^*_G(Y_2) \to H^*_G(Y_1), \]
which we refer to as a \emph{persistent equivariant cohomology module}.
We also work with filtrations and persistent equivariant cohomology modules that are indexed by $\mathbb{R}$, in an analogous way.

Suppose that $X$ is a metric space equipped with an action $G\times X\to X$ of the group $G$ by isometries.
We can build a filtration on top of $X$ using Vietoris--Rips metric thickenings.
Indeed, the \emph{Vietoris--Rips metric thickening} $\vrm{X}{r}$ consists of all finitely-supported probability measures on $X$ whose support have diameter at most $r$~\cite{AAF}.
The space $\vrm{X}{r}$ inherits a $G$ action from the $G$ action on $X$.
Therefore, given an increasing sequence of scale parameters $r_1\le \ldots\le r_k$, we obtain a filtration
\[\vrm{X}{r_1} \hookrightarrow \vrm{X}{r_2} \hookrightarrow \ldots \hookrightarrow \vrm{X}{r_{k-1}} \hookrightarrow \vrm{X}{r_k}\]
and hence a persistent equivariant cohomology module
\[H^*_G(\vrm{X}{r_k}) \to H^*_G(\vrm{X}{r_{k-1}}) \to \ldots \to H^*_G(\vrm{X}{r_2}) \to H^*_G(\vrm{X}{r_1}).\]

As our primary example, in this paper we determine the persistent equivariant cohomology of the Vietoris--Rips metric thickenings of the circle.
Note that the circle $S^1$ is both a group and a metric space.
The group $S^1$ therefore acts on the Vietoris--Rips metric thickening $\vrm{S^1}{r}$.
The action is free for $r<\frac{2\pi}{3}$, but has fixed points of the form $\frac{1}{3}\delta_\theta+\frac{1}{3}\delta_{\theta+2\pi/3}+\frac{1}{3}\delta_{\theta+4\pi/3}$ once $r\ge\frac{2\pi}{3}$.
Here $\theta \in \mathbb{R}$ is any angle and $\delta_\theta$ is the Dirac delta mass at $e^{i \theta} \in S^1$.
Then more fixed points (of the form $\sum_{j=0}^4 \frac{1}{5}\delta_{\theta+2\pi j/5}$) appear once $r\ge\frac{4\pi}{5}$, and then even more fixed points (of the form $\sum_{j=0}^6 \frac{1}{7}\delta_{\theta+2\pi j/7}$) appear once $r\ge\frac{6\pi}{7}$, etc.
We determine the equivariant cohomology $H^*_{S^1}(\vr{S^1}{r})$ 
for all values of the scale parameter $r$, as well as the induced morphisms $H^*_{S^1}(\vr{S^1}{r'}) \to H^*_{S^1}(\vr{S^1}{r})$ for all pairs of values of the scale parameter $r\le r'$.

\begin{restatable}{theorem}{thmMain}
\label{thm:main}
The persistent equivariant cohomology of the Vietoris--Rips thickenings of the circle is
\[
H^*_{S^1}(\vrm{S^1}{r};\Z)\cong
\begin{cases}
\Z[u]/(m_0\cdot\ldots \cdot m_k\, u^{k+1}) & \text{if }\frac{2\pi k}{2k+1}\le r<\frac{2\pi(k+1)}{2k+3}\text{ for }k\in \{0,1,2,\ldots\} \\
\Z[u] & \text{if } r \ge \pi,
\end{cases}
\]
where $\deg(u)=2$, and where $m_j=2j+1$.
If $r\le r'$, and if 
\[
\tfrac{2\pi k}{2k+1}\le r<\tfrac{2\pi(k+1)}{2k+3} ~~ \textrm{ and } ~~ \tfrac{2\pi k'}{2k'+1}\le r'<\tfrac{2\pi(k'+1)}{2k'+3},
\]
then necessarily $k\le k'$, and the induced map
\[
\Z[u]/(m_0\cdot\ldots \cdot m_{k'}\, u^{k'+1}) \cong H^*_{S^1}(\vrm{S^1}{r'};\Z)\to H^*_{S^1}(\vrm{S^1}{r};\Z)\cong \Z[u]/(m_0\cdot\ldots \cdot m_k\, u^{k+1})
\]
is the projection induced by the identity on $\Z[u]$.
\end{restatable}

This map is well-defined since $k\le k'$ means that $m_0\cdot\ldots \cdot m_k\, u^{k+1}$ divides evenly into $m_0\cdot\ldots \cdot\, m_{k'} u^{k'+1}$.
If $r'\ge \pi$ then the induced map
\[
    \Z[u]\cong H^*_{S^1}(\vrm{S^1}{r'};\Z) \longrightarrow H^*_{S^1}(\vrm{S^1}{r};\Z)\cong \Z[u]/(m_0\cdot\ldots \cdot m_k\, u^{k+1})
\]
is the canonical projection, and if also $r\ge \pi$ then the induced map 
\[
    \Z[u]\cong H^*_{S^1}(\vrm{S^1}{r'};\Z) \longrightarrow H^*_{S^1}(\vrm{S^1}{r};\Z)\cong \Z[u]
\]
is the identity.

Given a space $Y$, cohomology is a classical invariant to measure the ``shape'' of the space.
When a space $Y$ admits an action by a group $G$, equivariant cohomology measures not only the space $Y$ but also properties of the group action.
How does this work when one begins not with a topological space $Y$, but instead with dataset $X$?
Datasets are typically discrete, and hence have no interesting topology upon first glance.
However, persistent cohomology (and Vietoris--Rips thickenings) are a way to first turn a dataset into a growing sequence of spaces, and then to compute the cohomological invariants of this filtration.
When a dataset $X$ admits an action by a group $G$, persistent equivariant cohomology measures not only the ``shape'' of the dataset, but also properties of the group action --- i.e.\, properties of the symmetries within the dataset.

Part of our motivation is set in the context of the more general problem of determining the (equivariant) homotopy type of $\vr{S^n}{r}$ and $\vrm{S^n}{r}$, the Vietoris--Rips complexes and metric thickenings of spheres.
The paper~\cite{AAF} shows that
\[
    \vrm{S^n}{r} \simeq \begin{cases}
        S^n & 0 \le r < r_n\\
        S^n* \frac{SO(n)}{A_{n+2}} & r = r_n,
    \end{cases}
\]
where $r_n$ is the diameter of the standard simplex $\Delta_{n+1}$ with vertices on $S^n$, and $A_{n+2} \subseteq SO(n)$ is the group of rotational symmetries of $\Delta_{n+1}$.
To be able to go beyond the constraints $0 \le r \le r_n$, it is likely that a more symmetry-preserving approach would be necessary. As a first step in this direction, we obtain the equivariant homotopy type in the case $n=1$.

In Section~\ref{sec:ec}, we give a self-contained introduction to equivariant cohomology, along with pointers to standard references with more information.
In Section~\ref{sec:pec}, we introduce persistent equivariant cohomology, while also describing related work along with motivating potential applications to datasets equipped with symmetries or approximate symmetries.
We describe an incorrect guess (and admittedly, our first guess) for $H^*_{S^1}(\vrm{S^1}{r};\Z)$ in Section~\ref{sec:guess}, in order to demonstrate that the $S^1$ action on $\vrm{S^1}{r}$ is \emph{not} the same as the most common $S^1$ action on odd spheres.
In Section~\ref{sec:pec-vr}, we outline the proof of Theorem~\ref{thm:main}, which consists of two steps.
First, in Section~\ref{sec:epc-S1join}, we use the Serre spectral sequence and the Gysin homomorphism to identify the equivariant cohomology of a particular $S^1$ action on an odd-dimensional sphere, realized as a join of multiple circles.
Second, in Section~\ref{ssec:S1-he}, we show that $\vrm{S^1}{r}$ is $S^1$-homotopy equivalent to this particular join of circles, thus completing the proof of Theorem~\ref{thm:main}.
We conclude with open questions in Section~\ref{sec:conclusion}.

\section{What is \ldots equivariant cohomology?}
\label{sec:ec}

We follow 
Tu's article ``\emph{What is \ldots equivariant cohomology?}''~\cite{tu2013equivariant}, which is part of the ``\emph{What is \ldots}'' series in the Notices of the American Mathematical Society.
For more complete references on equivariant cohomology, we refer the reader to~\cite{may1996equivariant,tom2006transandrepres,schwede,nlab:equivariant_cohomology}.
If $G$ is a topological group acting on a space $Y$, then the first attempt at a definition of the equivariant cohomology might be to define it as the standard cohomology of the quotient space $Y/G$.
However, this behaves poorly, in particular when the action of $G$ on $Y$ has fixed points.
Instead, we find a contractible space $EG$ on which the group $G$ acts \emph{freely}, i.e.\ without fixed points.
Then $EG\times Y$ still has the same homotopy type as $Y$.
Furthermore, the diagonal action of $G$ on $EG\times Y$ (defined by $g\cdot(z,y)=(g\cdot z,g\cdot y)$) is guaranteed to be free, even if the action of $G$ on $Y$ alone is not.
The \emph{homotopy quotient} $EG \times_G Y$ is defined to be the quotient of $EG\times Y$ by the free diagonal action of $G$.
Then, we define the \emph{equivariant cohomology} of the action of $G$ on $Y$ to be $H_G^*(Y) \coloneqq H^*(EG \times_G Y)$.
In other words, the equivariant cohomology $H_G^*(Y)$ is the standard (singular) cohomology of the homotopy quotient $EG \times_G Y$.

Given a $G$-equivariant map $Y\to Y'$, we obtain a $G$-equivariant map $EG\times Y \to EG\times Y'$ and a map $EG \times_G Y \to EG \times_G Y'$.
By the contravariance of cohomology, we get an induced morphism $H^*(EG \times_G Y')\to H^*(EG \times_G Y)$, or $H_G^*(Y') \to H_G^*(Y)$.
Also, given a subgroup $G$ of $G'$, the inclusion $G\hookrightarrow G'$ induces a morphism $H_{G'}^*(Y) \to H_G^*(Y)$.

In general, suppose we have a $G$-space $X$ and a $G'$-space $X'$ for two groups $G$ and $G'$, and suppose we have a group homomorphism $\varphi \colon G \to G'$ and a continuous map $f \colon X \to X'$ satisfying $f(g \cdot x) = \varphi (g) \cdot f(x)$.
Then this induces a map $H^*_{G'}(X') \to H^*_G(X)$~\cite[Section~1.3]{anderson2011}.
In particular, for a fixed $G$, the equivariant cohomology $H^*_G(-)$ is a contravariant functor on the category of $G$-spaces.
Also, for a given $G'$-space $X$, any group homomorphism $\varphi \colon G \to G'$ induces a homomorphism $H^*_{G'}(X) \to H^*_G(X)$, where $X$ is also a $G$-space via $\varphi$.

A case of particular importance is when the space is $\textrm{pt}$, a single point.
The homotopy quotient $(EG\times \pt)/G=EG/G$ has a more recognizable name $BG$, the \emph{classifying space} of the group $G$.
The map $EG\to BG$ appears in homotopy theory as a \emph{universal principal $G$-bundle}, of which every principal $G$-bundle is a pullback~\cite{milnor1974characteristic,hatcher2003vector,RudolphSchmidt2017}.
Since any $G$-space $Y$ has a $G$-equivariant map $Y\to \pt$, we obtain a morphism $H^*(BG)=H^*_G(\pt)\to H^*_G(Y)$.
Hence, the equivariant cohomology $H^*_G(Y)$ is a module over the ring $H^*(BG)$.
Furthermore, given a $G$-equivariant map $Y\to Y'$, the induced map $H_G^*(Y') \to H_G^*(Y)$ is a $H^*(BG)$-module morphism.

At the end of the 1930s, Paul Althaus Smith initiated the study of transformation groups using tools from algebraic topology.
He proved several results about the fixed points of a finite $p$-group action.
Results of this type became known as Smith Theory~\cite{smith_theory}.
Interested in Smith Theory, Armand Borel organized in 1958--1959 a seminar on transformation groups, the central object of which was what is now known as the \emph{Borel construction}, $EG\times_G X$, and its cohomology~\cite{Borel_memoir_Goresky}.
Typical problems, which were considered, were centered around the topological properties of a compact lie group $G$, a $G$-space $X$, its fixed points, its orbits, and the orbit space. 

An important tool in equivariant cohomology is the localization theorem of Borel--Atiyah--Segal type, formulated in the 1960s.
It establishes the relationship between localized equivariant cohomology of a nice-enough $G$-space $X$, where $G$ is a compact lie group, and localized equivariant cohomology of its fixed point spaces~\cite{Hsiang1975}.
In particular, for the case when $G$ is $(S^1)^k$ or $\mathbb Z_p^k$, for a prime $p$, the localization theorem allows to determine the equivariant cohomology of a space in terms of the (non-equivariant) cohomology of the fixed point space and gives a criterion to establish the existence of fixed points.
Another powerful consequence of the localization theorem is Berline--Vergne--Atiyah--Bott localization formula, which was formulated in the 1980s simultaneously by Michael Atiyah and Raoul Bott~\cite{AtiyahBott1984} and by Mich\'ele Vergne and Nicole Berline~\cite{BerlineVergne1983}.
Existing now in a variety of generalizations, localization formulae enable the integration of an equivariant form over a $G$-space by integrating equivariant forms over fixed points.
This equivariant integration technique found multiple uses in mathematics and physics~\cite{AndersonFulton2024, Cordes_et_al_1996,  Pestun_2017, Szabo2000, Vergne2007}. 

On the other hand, equivariant cohomology has proven to be one of the main tools in topological combinatorics.
In this field, broadly speaking, problems originating from combinatorial or discrete geometry are often solved using topological methods.
One of the earliest applications of such methods to a combinatorial problem was the solution of Kneser's conjecture by Lov\'asz~\cite{Lovasz1978} in 1978.
Other notable examples include Tverberg-type problems~\cite{BSS81, BMZ15, Ozaydin87, Tverberg66}, mass partition problems~\cite{Avis84, BFHZ18, Grunbaum60, Hadwiger66, Ramos96, RPS22}, the Nandakumar--Ramana Rao problem \cite{akopyan2018convex, BS23+, blagojevic2014convex, karasev2014convex}, necklace splitting problems~\cite{AA89, Alon87, HR65}, and Knaster-type questions~\cite{KV10, Knaster47, Volovikov93}, among others.

The original combinatorial or geometric problem is typically reduced to proving that a certain equivariant map $f \colon X \longrightarrow_G Y$ has an image in $Y$ which non-trivially intersects a certain subspace $Z \subseteq Y$.
The goal is then to show non-existence of an equivariant map $X \longrightarrow_G Y \setminus Z$.
For a gentle introduction to this method and a plethora of applications, we refer to the book by Matou\v sek~\cite{matousek2003using} and to~\cite{Zivaljevic1996, Zivaljevic1998}.

The non-existence of equivariant maps is a general problem that can be approached through various methods.
These include equivariant obstruction theory~\cite{tom1987transgroups}, where obstructions to the existence of equivariant maps live in the $G$-equivariant cohomology of $X$, possibly with twisted coefficients.
Some applications of this method are found in~\cite{BFHZ16, BMZ15, BS23+, BZ14, FMSS24}. 

Another approach to showing non-existence of equivariant maps is via computations of a cohomological ideal-valued index called the Fadell--Husseini index~\cite{Fadell1988}.
(See also \cite[Section~6]{matousek2003using} for applications of a numerical-valued index.)
Namely, the Fadell--Husseini index of a $G$-space $X$ is an ideal $\textrm{ind}_G(X) = \ker(H^*_G(\pt) \to H^*_G(X))$, and an existence of an equivariant map $X \longrightarrow_G Y \setminus Z$ would imply $\textrm{ind}_G(Y \setminus Z) \subseteq \textrm{ind}_G(X)$.
This provides another obstruction to the existence of such a map.
Applications of the Fadell--Husseini index, along with related computations, can be found in~\cite{BLZ15, BPSZ19, BS18, KV10, Volovikov93}.
In both of these methods, equivariant cohomology plays a central role.

\section{What is \ldots \emph{persistent} equivariant cohomology?}
\label{sec:pec}

First, let us describe persistent cohomology in the non-equivariant setting.
A \emph{filtration of spaces} $\{Y_r\}_{r\in \R}$ is a functor from $(\R,\le)$ to topological spaces.
In other words, for each $r\in \R$ we have a space $Y_r$, and for each $r\le r'$ we have a morphism $Y_r \to Y_{r'}$ satisfying certain commutativity conditions.\footnote{
The mentioned commutativity conditions are $h_{r,r}=id_{Y_r}$ and $h_{r',r''}\circ h_{r,r'}=h_{r,r''}$ for all $r\le r'\le r''$.
}
Let $\F$ be a field.
If we apply the cohomology functor $H^*(\cdot;\F)$ to a filtration, then we obtain a cohomology group $H^*(Y_r)$ for each $r\in \R$, and by contravariance we obtain a morphism $H^*(Y_{r'}) \to H^*(Y_r)$ for each $r\le r'$.
We refer to this as the data of a \emph{persistent cohomology module}~\cite{EdelsbrunnerHarer,zomorodian2005computing,de2011dualities}.
See Weinberger's article ``\emph{What is \ldots persistent homology?}''~\cite{weinberger2011persistent} 
for connections from persistent (co)homology to geometric group theory, to the cohomological dimension of a discrete group, to the loop space of a Riemannian manifold, and to the amount nullhomotopic loops need to be expanded in order to be contracted.
See~\cite{polterovich2020topological} for connections between persistent homology and symplectic geometry.

Now, we move to the equivariant setting.
We refer the reader to Maia Fraser's paper~\cite{fraser2015contact} for an early appearance of persistent equivariant cohomology, and to the papers~\cite{Chacholski2023, Conti2022, Frosini2014, Frosini2016, Zhang2019} for other directions in equivariant persistence.
Let $G$ be a topological group.
A \emph{filtration of $G$-spaces} $\{Y_r\}_{r\in \R}$ is a functor from $(\R,\le)$ to $G$-spaces.
In other words, for each $r\in \R$ we have a $G$-space $Y_r$, and for each $r\le r'$ we have a $G$-equivariant morphism $Y_r \to Y_{r'}$ satisfying certain commutativity conditions.\footnote{
\emph{$G$-equivariant} means that $h_{r,r'\colon} Y_r \to Y_{r'}$ satisfies $h_{r,r'}(g\cdot y) = g\cdot h_{r,r'}(y)$ for all $g\in G$ and $y\in Y_r$.
}
If we apply the equivariant cohomology functor $H^*_G(-;\F)$ to a filtered $G$-space, then we obtain an equivariant cohomology group $H^*_G(Y_r)$ for each $r\in \R$, and by contravariance we obtain a morphism $H^*_G(Y_{r'}) \to H^*_G(Y_r)$ for each $r\le r'$, again satisfying commutativity conditions.
We refer to this data as a \emph{persistent equivariant cohomology module}.

Let $X$ be a metric space.
Recall the \emph{Vietoris--Rips metric thickening}~\cite{AAF} consists of all finitely-supported probability measures on $X$ whose support have diameter at most $r$, namely
\[\vrm{X}{r}\coloneqq \left\{\sum_{j=0}^n a_j \delta_{x_j}~\Bigg|~n\ge 0,\ a_j\ge 0,\ \sum_j a_j=1,\ \diam(\{x_0,\ldots,x_n\})\le r\right\}.\]
The space $\vrm{X}{r}$ is equipped with an optimal transport metric~\cite{vershik2013long,villani2003topics,villani2008optimal}.
Though the reader may be more familiar with the Vietoris--Rips simplicial complex, the Vietoris--Rips metric thickening and the Vietoris--Rips simplicial complex are homeomorphic when $X$ is finite and homotopy equivalent when $X$ is discrete; see Propositions~6.2 and~6.6 of~\cite{AAF}.

We have a filtration $\{\vrm{X}{r}\}_{r\in \R}$ of Vietoris--Rips metric thickenings, since for all $r\le r'$ we have an inclusion $\vrm{X}{r}\hookrightarrow\vrm{X}{r'}$.
If $X$ is equipped with an action of a group $G$ by isometries, then each metric thickening $\vrm{X}{r}$ is equipped with an induced $G$ action defined by $g\cdot(\sum_j a_j \delta_{x_j})=\sum_j a_j \delta_{g\cdot x_j}$.
Hence, we obtain an associated persistent equivariant cohomology module $\{H_G^*(\vrm{X}{r})\}_{r\in \R}$.

Suppose there exists a value $r_\infty\in\R$ such that $\vrm{X}{r}$ is contractible for all $r\ge r_\infty$.
This happens, for example, if $X$ is bounded and $r_\infty=\diam(X)$.
Famously, Rips proved that this alternatively happens if $X$ is a $\delta$-hyperbolic group equipped with the word metric and $r_\infty=4\delta$~\cite{Gromov1987}.
If such an $r_\infty$ exists, then the map $Y\to\pt$ mentioned in the last paragraph of Section~\ref{sec:ec} (showing that the equivariant cohomology $H^*_G(Y)$ is a module over the ring $H^*(BG)$) has a realization up to homotopy as the map $\vrm{X}{r}\to\vrm{X}{r_\infty}\simeq\pt$.
So, in the setting where the Vietoris--Rips metric thickenings $\vrm{X}{r}$ eventually become $G$-contractible as $r$ increases, persistent equivariant cohomology is a way to break up the map $\vrm{X}{r}\to\pt$ into a sequence of steps 
\[\vrm{X}{r_1} \hookrightarrow \vrm{X}{r_2} \hookrightarrow \vrm{X}{r_3} \hookrightarrow \ldots \hookrightarrow \vrm{X}{r_\infty} \simeq \pt,\]
enabling the study of how the equivariant cohomology approaches that of $H^*(BG)=H^*_G(\pt)$ as the scale parameter increases.

If the action of $G$ on $\vrm{X}{r_\infty}\simeq\pt$ is also free then $EG$ can be taken to be $\vrm{X}{r_\infty}$.
This was done by Rips when $G=X$ is a torsion-free $\delta$-hyperbolic group equipped with the (discrete) word metric to obtain finite-dimensional classifying spaces $BG=EG/G$~\cite[III.G.3, Theorem~3.21, Page~468]{bridson2011metric}; see also~\cite{Gromov1987,Gromov,ghys2013groupes}.
For further connections between Vietoris--Rips thickenings and geometric group theory, we refer the reader to the paper~\cite{zaremsky2022bestvina} 
and talk~\cite{Zaremsky-VRtalk2021} 
by Zaremsky, which explain how a finitely generated group $G$ is of type $F_n$ if and only if its Vietoris--Rips thickenings are essentially $(n-1)$-connected; see also~\cite{alonso1992combings,alonso1994finiteness,brown1987finiteness}.

As mentioned in the introduction, just as persistent cohomology has found applications in data science in determining the ``shape'' of data~\cite{Carlsson2009,EdelsbrunnerHarer,edelsbrunner2000topological,robins2000computational}, we hope that persistent equivariant cohomology will find applications when a dataset is equipped with symmetries or approximate symmetries.
Indeed, such a symmetry can be encoded via a group action  (or approximate group action) on a dataset.
We refer the reader to~\cite{graczyk2022model,huang2024approximately,korman2015probably,mgp_approx_symm_sig_06,bergomi2019towards,bocchi2023finite} for work on finding symmetries or approximate symmetries in a dataset, which can be used for example to inpaint or fill-in missing regions of the data.

A particular example of data with symmetries occurs in crystalline structures in materials science.
A crystal might consist of a periodic or near-periodic pattern of atoms.
Given a choice of fundamental domain for a crystal, one can reduce the study of the entire crystal to the easier task of studying only the atoms inside a single fundamental domain, remembering the identifications on the boundary of that fundamental domain.
However, when symmetries are only approximate, it is not clear what the correct fundamental domain should be.
One would like a data analysis technique that is stable to perturbations in the data, and which therefore should be stable to the choice of two different reasonable fundamental domains~\cite{edelsbrunner2021density}.
For work along these lines in the direction of persistent homology, we refer the reader to~\cite{HeissPreprint,HeissTalk,onus2022quantifying}.

In the paper~\cite{GH-BU-VR}, we studied the invariant
\[c_{1,2k+1}\coloneqq \inf\{ r > 0~|~\text{there exists a }\Z/2\text{ map }S^{2k+1} \to VR(S^1;r) \},\]
and connected it to Gromov--Hausdorff distances between spheres (see also~\cite{lim2022gromov,harrison2023quantitative}).
We are also interested in variants thereof, such as 
\[\inf\{ r > 0~|~\text{there exists an }S^1\text{ map }S^{2k+1} \to VR(S^1;r) \},\] 
which of course might depend on the $S^1$ action one places on the odd sphere $S^{2k+1}$.

Let $G$ be a group acting properly and by isometries on a metric space $X$.
In~\cite{AdamsHeimPeterson}, the authors study when $\vrm{X}{r}/G$ is homeomorphic to $\vrm{X/G}{r}$, i.e., when the quotient of the Vietoris--Rips thickening is the Vietoris--Rips thickening of the quotient.
In this setting, if the induced action of $G$ on $\vrm{X}{r}$ is free, then the equivariant cohomology $H^*_G(\vrm{X}{r})$ is equal to the standard cohomology $H^*(\vrm{X}{r}/G)=H^*(\vrm{X/G}{r})$.

See~\cite{varisco2021equivariant} for an equivariant version of Bestvina--Brady Morse theory, applied to Vietoris--Rips complexes.

\section{An incorrect guess for $H^*_{S^1}(\vrm{S^1}{r};\Z)$}
\label{sec:guess}

It is well known that the classifying space of the circle is $BS^1=\CP^\infty$, the infinite-dimensional complex projective space.
One way to see this is as follows.
We view each odd-dimensional sphere $S^{2k+1}$ as the unit sphere in $\C^{k+1}$.
The circle $S^1\subseteq \C$ acts on $\C^{k+1}$ by scalar multiplication $\lambda\cdot(z_1,\ldots,z_{k+1})=(\lambda z_1,\ldots,\lambda z_{k+1})$.
The induced action of $S^1$ on $S^{2k+1}$ is free, and the complex projective space $\CP^k$ is defined to be the quotient of $S^{2k+1}$ under this group action.
Furthermore, each inclusion $S^{2k+1}\hookrightarrow S^{2k'+1}$ for $k\le k'$, defined via $(z_1,\ldots,z_{k+1}) \mapsto (z_1,\ldots,z_{k+1},0,\ldots,0)$, is $S^1$-equivariant.
If we let $S^\infty$ be the union $S^\infty=\cup_{k=0}^\infty S^{2k+1}$, then this is a contractible space on which $S^1$ acts freely, and hence can be chosen to be the total space $ES^1$.
Then $\CP^\infty=\cup_{k=0}^\infty \CP^{k}$ is the quotient of $ES^1=S^\infty$ under this free action by $S^1$, and so $BS^1=\CP^\infty$ is the classifying space of the circle.

So, consider the $S^1$-equivariant ``odd sphere filtration''
\begin{equation}
\label{eq:odd-sphere}
    S^1\hookrightarrow S^3\hookrightarrow S^5\hookrightarrow S^7\hookrightarrow \ldots \hookrightarrow S^\infty,
\end{equation}
and apply $H_{S^1}^*$ to obtain the persistent equivariant cohomology module
\[
    H_{S^1}^*(S^\infty)\longrightarrow\ldots\longrightarrow H_{S^1}^*(S^7)\longrightarrow H_{S^1}^*(S^5)\longrightarrow H_{S^1}^*(S^3)\longrightarrow H_{S^1}^*(S^1).
\]
Since $S^\infty$ is contractible, the first term in this module is $H_{S^1}^*(S^\infty) \cong H^*_{S^1}(pt) = H^*(BS^1) \cong H^*(\CP^\infty)$, the standard cohomology of the classifying space of the circle, $\CP^\infty$.
And, since the action of $S^1$ on each odd sphere $S^{2k+1}$ is free, the equivariant cohomology $H_{S^1}^*(S^{2k+1})$ is simply the standard cohomology of the quotient $\CP^k$~\cite[I, Lemma 4.9]{Tu2020}.
So, the terms of this persistent equivariant cohomology module are
\[
H^*_{S^1}(S^{2k+1};\Z)=H^*(\CP^k;\Z)\cong
\begin{cases}
\Z[u]/(u^{k+1}) & \text{if }k<\infty \\
\Z[u] & \text{if } k=\infty,
\end{cases}
\]
where $\deg(u)=2$.
Furthermore, for positive integers $k\le k'$, inclusion of spheres induces the canonical projection
\[
    \Z[u]/(u^{k'+1})\cong H^*_{S^1}(S^{2k'+1};\Z) \longrightarrow H^*_{S^1}(S^{2k+1};\Z)\cong \Z[u]/(u^{k+1})
\]
in equivariant cohomology.
If $k'=\infty$ then this map is also the canonical projection $\Z[u] \longrightarrow \Z[u]/(u^{k+1})$.

The homotopy types of the Vietoris--Rips complex or the Vietoris--Rips metric thickenings of the circle are known be the circle, the 3-sphere, the 5-sphere, the 7-sphere, \ldots, indeed obtaining all odd spheres as the scale increases~\cite{AA-VRS1,moyVRmS1}.
More explicitly, by~\cite{moyVRmS1} we have homotopy equivalences
\[
\vrm{S^1}{r}\simeq\begin{cases}
S^{2k+1} & \text{if }\frac{2\pi k}{2k+1}\le r<\frac{2\pi(k+1)}{2k+3} \\
\pt & \text{if } r \ge \pi.
\end{cases}
\]
Hence we have a filtration
\[\begin{tikzcd}
\vrm{S^1}{0} & \vrm{S^1}{\frac{2\pi}{3}} & \vrm{S^1}{\frac{4\pi}{5}} & \vrm{S^1}{\frac{6\pi}{7}} & \ldots & \vrm{S^1}{\pi} \\
{S^1} & {S^3} & {S^5} & {S^7} && {S^\infty}
\arrow[hook, from=1-1, to=1-2]
\arrow["\simeq"', from=1-1, to=2-1]
\arrow[hook, from=1-2, to=1-3]
\arrow["\simeq"', from=1-2, to=2-2]
\arrow[hook, from=1-3, to=1-4]
\arrow["\simeq"', from=1-3, to=2-3]
\arrow[hook, from=1-4, to=1-5]
\arrow["\simeq"', from=1-4, to=2-4]
\arrow[hook, from=1-5, to=1-6]
\arrow["\simeq"', from=1-6, to=2-6]
\end{tikzcd}\] 
Seemingly, there is a close analogy between the Vietoris--Rips metric thickenings of the circle and the odd-sphere filtration~\eqref{eq:odd-sphere}.
This might lead one to conjecture that the persistent equivariant cohomology 
$H^*_{S^1}(\vrm{S^1}{r};\Z)$ of the Vietoris--Rips metric thickenings of the circle is given by $\Z[u]/(u^{k+1})$ if $\frac{2\pi k}{2k+1}\leq r<\frac{2\pi(k+1)}{2k+3}$, where $\deg(u)=2$.
However, this guess is incorrect.
A crucial difference is that even though $\vrm{S^1}{r}$ is homotopy equivalent to an odd sphere for $r<\pi$, the $S^1$ action on $\vrm{S^1}{r}$ is not free for $r\ge \frac{2\pi}{3}$, as we explain in Section~\ref{ssec:S1-he}.
Furthermore, more and more fixed points arise as $r$ increases beyond each critical value $\frac{2\pi}{3}$, $\frac{4\pi}{5}$, $\frac{6\pi}{7}$, \ldots, and
these fixed points have a large effect on the persistent equivariant cohomology.
Indeed, it turns out that 
\[
    H^*_{S^1}(\vrm{S^1}{r};\Z)\cong \Z[u]/(1\cdot 3\cdot \ldots\cdot (2k+1)u^{k+1})
\]
for $\frac{2\pi k}{2k+1}\le r<\frac{2\pi(k+1)}{2k+3}$ (see Theorem~\ref{thm:main}), as we will prove in the following sections.

\section{Persistent equivariant cohomology of $\vrm{S^1}{r}$}
\label{sec:pec-vr}

Let $S^1=\{z\in\C~|~\|z\|=1\}$ be the unit circle in the complex plane.
The circle group is the set $S^1$ with the operation of complex multiplication; we write elements of this group as $\lambda$ or as $\lambda_\theta = e^{i \theta}$ to specify an element by an angle $\theta \in \mathbb{R}$.
The Vietoris--Rips metric thickening of the circle $\vrm{S^1}{r}$ consists of finitely-supported probability measures on $S^1$ with support of diameter at most $r$.
We write measures in $\vrm{S^1}{r}$ as $\sum_{j=0}^n a_j \delta_{\theta_j}$, where $\delta_{\theta_j}$ is the Dirac delta measure at the point $e^{i \theta_j}$.
We will let the circle group act on $\vrm{S^1}{r}$ by rotating measures on the circle:
\[
\lambda_{\theta} \cdot \sum_{j=0}^n a_j \delta_{\theta_j} = \sum_{j=0}^n a_j \delta_{\theta+\theta_j}.
\]

We recall the statement of Theorem~\ref{thm:main}.

\thmMain*

In particular, if $r'\ge \pi$, then the induced map 
\[
\Z[u]=H^*_{S^1}(\vrm{S^1}{r'};\Z)\to H^*_{S^1}(\vrm{S^1}{r};\Z)=\Z[u]/(m_0\cdot\ldots \cdot m_k\, u^{k+1})
\]
is the canonical projection, and if also $r\ge \pi$ then this induced map $\Z[u]=H^*_{S^1}(\vrm{S^1}{r'};\Z)\to H^*_{S^1}(\vrm{S^1}{r};\Z)=\Z[u]$ is the identity.\\

How do we show the theorem?
First, in Theorem~\ref{thm:equivariant-cohomology}, we describe the equivariant cohomology of the $(2k-1)$-dimensional sphere $S^{2k-1}_{m_1,\dots,m_k}$, equipped with a particular $S^1$ action.
By Remark~\ref{rem: naturality}, this description is functorial with respect to inclusion $S^{2k-1}_{m_1,\dots,m_k} \hookrightarrow S^{2k+1}_{m_1,\dots,m_{k+1}}$.
Then, in Theorem~\ref{thm:S^1-equivalence}, we show that for $\frac{2\pi k}{2k+1}\le r<\frac{2\pi(k+1)}{2k+3}$, the metric thickenings $\vrm{S^1}{r}$ are $S^1$-homotopy equivalent to $S^{2k-1}_{m_1,\dots,m_k}$ such that the $S^1$-diagram
\begin{equation*}
    \begin{tikzcd}
        \vrm{S^1}{r} \arrow[r, "\simeq"] \arrow[d, hook] & S^{2k-1}_{m_1,\dots,m_k} \arrow[d, hook]\\
        \vrm{S^1}{r'} \arrow[r, "\simeq"]  & S^{2k+1}_{m_1,\dots,m_{k+1}}
    \end{tikzcd}
\end{equation*}
commutes, where $\frac{2\pi (k+1)}{2k+3}\le r' <\frac{2\pi(k+2)}{2k+5}$.
When combined, Theorems~\ref{thm:S^1-equivalence} and~\ref{thm:equivariant-cohomology}, together with the Remark~\ref{rem: naturality}, give the proof of Theorem~\ref{thm:main}.\\

Regarding different cohomology coefficients, we note the following.
If we employ Remark~\ref{rem: coefficients} instead of Theorem~\ref{thm:equivariant-cohomology}, we get
\[
    H^*_{S^1}(\vrm{S^1}{r};\F)\cong
    \begin{cases}
    \F[u]/(u^{k+1}) & \text{if }\frac{2\pi k}{2k+1}\le r<\frac{2\pi(k+1)}{2k+3}\text{ for }k\in \{0,1,2,\ldots\} \\
    \F[u] & \text{if } r \ge \pi,
    \end{cases}
\]
where $\F$ is $\Q$, $\R$ or $\F_2$.
This isomorphism is also functorial with respect to the parameter $r$.

\section{Equivariant cohomology of $S^1_1*S^1_3*S^1_5*S^1_7*\ldots$}
\label{sec:epc-S1join}

Let $S^1_{2k+1}=S^1$ be the circle equipped with the following $S^1$ action:
$\lambda\cdot z=\lambda^{2k+1}z$ for all $\lambda\in S^1$ and $z\in S^1_{2k+1}$.
We determine the Borel equivariant cohomology of the diagonal $S^1$ action on $S^1_1*S^1_3*S^1_5*\ldots*S^1_{2k+1}$, defined by $\lambda\cdot(z_1,z_3,\ldots,z_{2k+1})=(\lambda z_1,\lambda^3 z_3,\ldots,\lambda^{2k+1} z_{2k+1})$.
We obtain
\[
H^*_{S^1}\left(S^1_1*S^1_3*S^1_5*\ldots*S^1_{2k+1};\Z\right)
\cong
\Z[u]/(1\cdot3\cdot\ldots\cdot (2k+1)\, u^{k+1}),
\]
where $\deg(u)=2$.

More generally, let $m_1, \dots, m_k$ be \emph{any} integers and let $S^{2k-1}_{m_1, \dots, m_k}$ denote the $(2k-1)$-sphere
\[
S^{2k-1}_{m_1, \dots, m_k} = S^1_{m_1} * \dots * S^1_{m_k}
\]
endowed with the diagonal action of $S^1$.
Namely, $S^1$ acts on the $i$-th copy $S^1_{m_i}$ by
\[
S^1 \times S^1_{m_i} \longrightarrow S^1_{m_i}, \quad (\lambda, z) \longmapsto \lambda^{m_i} \cdot z.
\]

\begin{theorem}
\label{thm:equivariant-cohomology}
The equivariant cohomology of $S^{2k-1}_{m_1, \dots, m_k}$ is
\[
H_{S^1}^*(S^{2k-1}_{m_1,\dots,m_k};\Z) \cong 
\begin{cases}
\Z[u]/(m_1\dots m_k\, u^k) & \text{if } m_1 \cdot\ldots\cdot m_k \neq 0\\
\Z[u,a]/(a^2) & \text{if } m_1 \cdot\ldots\cdot m_k = 0,
\end{cases} 
\]
where $|u| =2$ and $|a| = 2k-1$.
\end{theorem}
\begin{proof}

Let us denote by $\complex{m}$ the space $\C$ endowed with an $S^1$-action 
\[
S^1 \times \C \longrightarrow \C, \quad (\lambda, z) \longmapsto \lambda^m z.
\]
In particular, $S^1_m \subseteq \complex{m}$ is an invariant subspace.
We endow $\complex{m}$ with an orientation (as a 2-dimensional real vector space) by proclaiming the basis $1, i$ to be positively oriented.
This is the so-called \emph{canonical orientation} of the complex vector space~\cite
[Lemma~14.1]{Milnor1974}.
Accordingly, let
\begin{equation*}
\xi_{m}:~~ \complex{m} \longrightarrow ES^1 \times_{S^1} \complex{m} \longrightarrow BS^1
\end{equation*}
denote a complex line bundle.
The canonical orientation of $\complex{m}$ induces an orientation on the bundle $\xi_{m}$, where a positively oriented basis of any fiber $F \approx \C$ is of the form $v, iv$, for some non-zero vector $v \in F$.

Notice that $\xi_{1}$ is isomorphic to the canonical line bundle over $BS^1 = \C P^{\infty}$, so its first Chern class $c_1(\xi_{1})\in H^2(BS^1;\Z)$ is a generator of integral cohomology of $BS^1$~\cite
[Theorem~3.2(d)]{hatcher2003vector}.
Thus we can write 
\[
H^*(BS^1;\Z) \cong \Z[u]
\]
for $u \coloneqq c_1(\xi_{1}) \in H^2(BS^1;\Z)$.

For an integer $m \ge 1$, the isomorphism between $\complex{m}$ and $\complex{1}^{\otimes m} \cong \complex{1} \otimes_{\C} \dots \otimes_{\C} \complex{1}$ given by
\[
\complex{m} \longrightarrow \complex{1}^{\otimes m}, \quad z \longmapsto z \otimes 1 \otimes \dots \otimes 1
\] 
is $S^1$-equivariant and orientation preserving, with the diagonal action on the tensor product.\footnote{
Indeed, let $(\lambda,z)\in S^1\times \complex{m}$.
Mapping to $\complex{m}$ then $\complex{1}^{\otimes m}$ yields $\lambda^m z \otimes 1 \otimes \dots \otimes 1$, whereas mapping to $S^1\times \complex{1}^{\otimes m}$ then  $\complex{1}^{\otimes m}$ yields $\lambda z \otimes \lambda \otimes \dots \otimes \lambda$, which are equal in $\complex{1}^{\otimes m}$.
}
On the other hand, for $m \le -1$ the conjugation-induced isomorphism
\[
\complex{m} \longrightarrow \complex{1}^{\otimes |m|}, \quad z \longmapsto \overline{z}\otimes 1 \otimes \dots \otimes 1
\]
is $S^1$-equivariant, but orientation reversing.
Therefore, for any integer $m$, there is an orientation-preserving isomorphism of complex line bundles 
$\xi_{m} \cong \xi_{1}^{\otimes m}$.
For $m<0$ we interpret $\xi_{1}^{\otimes m}$ to be the conjugated bundle $\overline{\xi_1}^{\otimes |m|}$, and for $m=0$ the right hand side is the trivial complex line bundle $\xi_{0}$.

Now~\cite
[Proposition~3.10]{hatcher2003vector} implies  
\[
c_1(\xi_{m}) = c_1(\xi_{1}^{\otimes m}) = mc_1(\xi_{1}) = mu \in H^2(BS^1;\Z).
\]
Moreover, due to~\cite
[Proposition~3.13(c)]{hatcher2003vector}, the first Chern class of a complex line bundle is equal to the Euler class of the underlying 2-dimensional vector bundle with canonical orientation:
\[
e(\xi_{m}) = c_1(\xi_{m}) = mu \in H^{2}(BS^1;\Z).
\]

Let us denote by $\complex{m_1, \dots, m_k}^k$ the $k$-dimensional complex vector space $\complex{m_1} \oplus \dots \oplus \complex{m_k}$ which is endowed with the diagonal $S^1$-action.
The canonical orientation of $\complex{m_1, \dots, m_k}^k$ is the one with preferred basis $e_1,ie_1, \dots, e_k, ie_k$, where $e_j$ is the $j$-th standard complex basis vector for each $1 \le j \le k$.
Then we can denote by
\begin{equation*}
\xi_{m_1, \dots, m_k}:~~ \complex{m_1, \dots, m_k}^k \longrightarrow ES^1 \times_{S^1} \complex{m_1, \dots, m_k}^k \longrightarrow BS^1
\end{equation*}
the $k$-dimensional complex vector bundle (canonically orientated) and by 
\begin{equation}
\label{eq:sphere bundle}
S(\xi_{m_1, \dots, m_k}):~~ S(\complex{m_1, \dots, m_k}^k) \longrightarrow ES^1 \times_{S^1} S(\complex{m_1, \dots, m_k}^k) \xrightarrow{~~\pi~} BS^1
\end{equation}
the corresponding sphere bundle.
Here $S(\complex{m_1, \dots, m_k}^k)$ denotes the unit sphere inside $\complex{m_1, \dots, m_k}^k$.
Thus, we are inclined to prove
\[
H^*_{S^1}(S(\complex{m_1, \dots, m_k}^k);\Z) = H^*(ES^1 \times_{S^1} S(\complex{m_1, \dots, m_k}^k);\Z) \cong \begin{cases}
\Z[u]/(m_1\dots m_k u^k) & m_1 \cdot\ldots\cdot m_k \neq 0\\
\Z[u,a]/(a^2) & m_1 \cdot\ldots\cdot m_k = 0
\end{cases}
\]
for $|u| = 2$ and $|a| = 2k-1$.

The isomorphism of complex bundles $\xi_{m_1, \dots, m_k} \cong \xi_{m_1} \oplus \dots \oplus \xi_{m_k}$ preserves the orientation of the underlying real bundles, and due to~\cite
[Proposition~3.13(b)]{hatcher2003vector} we have 
\[
e(\xi_{m_1, \dots , m_k}) = e\left(\bigoplus_{i=1}^k \xi_{m_i}\right) = e(\xi_{m_1}) \cup \dots \cup e(\xi_{m_k}) = m_1 \dots m_k\, u^k \in H^{2k}(BS^1;\Z).
\]
The Leray--Serre spectral sequence of the sphere bundle $S(\xi_{m_1, \dots , m_k})$ has $E_2$-term equal to
\[
E_2^{p,q} \cong H^p(BS^1; \mathcal{H}^q(S(\complex{m_1, \dots, m_k}); \Z)),
\]
where $\mathcal{H}^q(S(\complex{m_1, \dots, m_k}); \Z)$ is the local system of coefficients.
Since $\pi_1(BS^1) \cong \pi_1(\C P^{\infty})$ is trivial, the local system is trivial.
Moreover, since $H^*(BS^1;\Z)$ and $H^*(S(\complex{m_1, \dots, m_k}); \Z)$ are free and finite-dimensional in each degree, due to~\cite
[Proposition~5.6]{mccleary} we have an isomorphism of algebras
\begin{equation}
\label{eq:E2 is tensor product}
E_2^{*,*} \cong H^*(BS^1; \Z) \otimes H^*(S(\complex{m_1, \dots, m_k}); \Z).
\end{equation}
The product on the right hand side of~\eqref{eq:E2 is tensor product} is the bialgebra product generated by
\begin{equation}
\label{eq:bialgebra structure}
(v \otimes w) \cdot (s \otimes t) = (-1)^{|w||s|}~ vs \otimes wt.
\end{equation}

Recall that $H^*(BS^1; \Z) \cong \Z[u]$ with $|u| = 2$ and $H^*(S(\complex{m_1, \dots, m_k}); \Z) \cong \Z[a]/(a^2)$ with $|a| = 2k-1$, so there is no sign in the product~\eqref{eq:bialgebra structure} because the degree $|s|$ is always even.
We have
\begin{equation*}
E_2^{*,q} \cong \begin{cases}
    H^*(BS^1;\Z) & q=0\\
    H^*(BS^1;\Z) \otimes \Z \langle a \rangle & q=2k-1\\
    0 & {\rm otherwise},
\end{cases}
\end{equation*}
where we denote by $\Z \langle a \rangle$ a free $\Z$-module of rank one with generator $a$.
Therefore the only potentially non-zero differential of the spectral sequence is $\partial_{2k}:E_{2k}^{*,*} \to E_{2k}^{*,*}$, and we have $E_2^{*,*} \cong \dots \cong E_{2k}^{*,*}$ and $E_{2k+1}^{*,*} \cong \dots \cong E_{\infty}^{*,*}$.
Additionally, due to~\cite
[Theorem~5.9]{mccleary} we have
\begin{equation}
\label{eq:index in the spectral sequence}
\im (\partial_{2k}\colon E_{2k}^{*,2k-1} \to E_{2k}^{*,0}) = \ker \left(\pi^*\colon H^*(BS^1;\Z) \to H^*(ES^1 \times_{S^1} S(\complex{(m_1, \dots, m_k)}))\right),
\end{equation}
where $\pi^*$ is the map induced from the projection $\pi$ of the sphere bundle~\eqref{eq:sphere bundle}; see Figure~\ref{fig:spectral_sequence}.

\begin{figure}[h]
  \centering
  \includegraphics[width=0.8\textwidth]{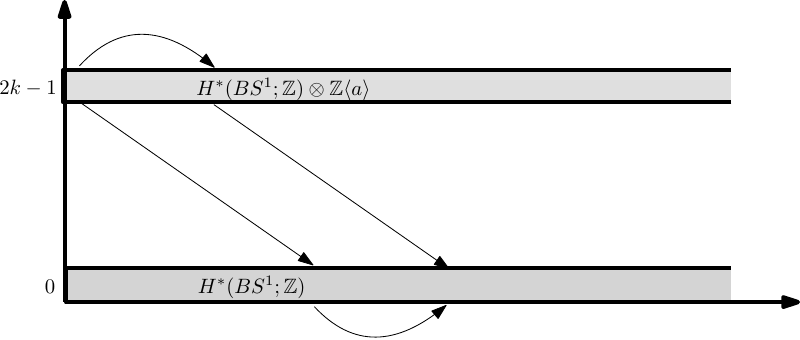}
  \caption{Leray--Serre spectral sequence of the sphere bundle.}
  \label{fig:spectral_sequence}
\end{figure}

From the Gysin sequence~\cite[Theorem~12.2]{milnor1974characteristic} of $\xi_{m_1, \dots, m_k}$ it follows that $\ker(\pi^*) \subseteq H^*(BS^1;\Z)$ is the principal ideal generated by the Euler class $e(\xi_{m_1, \dots, m_k}) = m_1 \dots m_k\, u^k \in H^*(BS^1;\Z)$.
So from~\eqref{eq:index in the spectral sequence} we have
\[
    \im (\partial_{2k}: E_{2k}^{*,2k-1} \to E_{2k}^{*,0}) = \ker (\pi^*) = (m_1 \dots m_k\, u^k) \subseteq H^*(BS^1;\Z).
\]
The map $\partial_{2k}\colon E_{2k}^{*,*} \to E_{2k}^{*,*}$ is a morphism of algebras, so in the view of~\eqref{eq:E2 is tensor product}, its image is the principal ideal generated by $\partial(1 \otimes a) \in E_{2k}^{0,2k}$, where 
\[
    1 \otimes a \in H^{0}(BS^1;\Z) \otimes H^{2k-1}(S(\complex{m_1, \dots, m_k});\Z) \cong E_{2k}^{2k-1, 0}.
\] 
Thus, we conclude $\partial_{2k}(1 \otimes a) = \pm m_1 \dots m_k\, u^k \in H^{2k}(BS^1;\Z)$, i.e., 
\[
    E_{2k+1}^{*,0} \cong \Z[u]/(m_1\dots m_k\, u^k).
\]

If $m_1 \cdot\ldots\cdot m_k \neq 0$ this is the only non-trivial part of $E_{2k+1}^{*,*}$, hence $E_{2k+1}^{*,*} \cong \dots \cong E_{\infty}^{*,*}$ are concentrated in the $0$-th row.
Therefore, in this case we conclude there is an algebra isomorphism
\[
    H^*(ES^1 \times_{S^1} S(\complex{m_1, \dots ,m_k}^k);\Z) \cong \Z[u]/(m_1\cdot\ldots\cdot m_k\, u^k),
\]
with $|u| =2$, as desired.

If $m_1 \cdot\ldots\cdot m_k =0$, then $\partial_{2k}=0$, so we have $E_2^{*,*} \cong \dots \cong E_{2k}^{*,*} \cong \dots \cong E_{\infty}^{*,*}$.
From~\eqref{eq:E2 is tensor product} and the fact that the $\Z$-modules $E_{\infty}^{r,s}$ are all free (hence projective), it follows that there is a $\Z[u]$-module isomorphism
\[
    H^*(ES^1 \times_{S^1} S(\complex{m_1, \dots ,m_k}^k);\Z) \cong \Z[u] \oplus (\Z[u] \otimes \Z\langle a\rangle)
\]
with $|u|=2$ and $|a| = 2k-1$.
Since 
\[
    1 \otimes a \in H^{2k-1}(ES^1 \times_{S^1} S(\complex{m_1, \dots ,m_k}^k);\Z) \cong \Z \langle 1 \otimes a \rangle
\]
is an element of odd degree, by graded-commutativity we have 
\[
    (1 \otimes a)^2 =(-1)^{|1 \otimes a|^2}(1 \otimes a)^2= - (1 \otimes a)^2.
\]
That is,
\[
    (1 \otimes a)^2 \in H^{4k-2}(ES^1 \times_{S^1} S(\complex{m_1, \dots ,m_k}^k);\Z) \cong \Z \langle u^{2k-1} \rangle
\]
has order 2 while living in a free $\Z$-module, so it has to be zero.
We can deduce the entire algebra structure of $H^*_{S^1}(S(\complex{m_1, \dots ,m_k}^k);\Z)$ now from its $\Z[u]$-module structure and commutativity:
\[
    (u^i \otimes a) \cdot (u^j \otimes a) = u^{i+j}(1 \otimes a)^2 = 0 \quad \text{ and } \quad (u^i \otimes a) \cdot u^j = u^j \cdot (u^i \otimes a) = u^{i+j} \otimes a.
\]
Namely, we obtain
\[
    H^*(ES^1 \times_{S^1} S(\complex{m_1, \dots ,m_k}^k);\Z) \cong \Z[u,a]/(a^2),
\]
with $|u| =2$ and $|a|=2k-1$.
This completes the proof of the theorem.
\end{proof}

\begin{remark}[Coefficients]
\label{rem: coefficients}
Regarding cohomology with field coefficients, we have
\[
    H^*_{S^1}(S(\complex{1,3, \dots, 2k-1}^k);\F)\cong \F[u]/(u^{k+1})
\]
where $\F$ is $\Q$, $\R$, or $\F_2$.
    
Indeed, Theorem~\ref{thm:equivariant-cohomology} implies that 
\[
    H^*_{S^1}(S(\complex{1,3, \dots, 2k-1}^k);\Z) = H^*(ES^1 \times_{S^1} S(\complex{1,3, \dots, 2k-1}^k);\Z)
\]
is finitely generated in each degree.
Therefore, by applying the universal coefficient theorem for cohomology~\cite[Ch.~5,~Sec.~5,~Thm.~10]{Spanier89},
we get a short exact sequence of $\Z$-modules 
\begin{equation*}
    0 \longrightarrow H^n_{S^1}(S(\complex{1,3, \dots, 2k-1}^k);\Z) \otimes \F \longrightarrow H^n_{S^1}(S(\complex{1,3, \dots, 2k-1}^k); \F) \longrightarrow \tor_1^{\Z}(H^{n+1}_{S^1}(S(\complex{1,3, \dots, 2k-1}^k);\Z); \F) \longrightarrow 0 
\end{equation*}
for any $n \ge 0$.

Since $\Q$ and $\R$ are flat $\Z$-modules, we have 
\[
    \tor_1^{\Z}(H^{n+1}_{S^1}(S(\complex{1,3, \dots, 2k-1}^k);\Z); \F)=0
\]
for $\F = \Q,\ \R$.
On the other hand,  
\[
    \tor_1^{\Z}(H^{n+1}_{S^1}(S(\complex{1,3, \dots, 2k-1}^k);\Z); \F_2) = \tor_1^{\Z}(H^{n+1}_{S^1}(S(\complex{1,3, \dots, 2k-1}^k);\Z); \Z/2)
\]
is the module of 2-torsion elements of $H^{n+1}_{S^1}(S(\complex{1,3, \dots, 2k-1}^k))$.
However, by Theorem~\ref{thm:equivariant-cohomology},
\[
    H^{n+1}_{S^1}(S(\complex{1,3, \dots, 2k-1}^k)) \cong
    \begin{cases}
    0 & n \textrm{ even},\\
    \Z & n \textrm{ odd and } 0 \le n < 2k-1,\\
    \Z/(1\cdot 3 \dots (2k-1)) & n \textrm{ odd and } n \ge 2k-1
    \end{cases}
\]
has no 2-torsion.
Therefore, 
\[
    \tor_1^{\Z}(H^{n+1}(\vrm{S^1}{r};\Z); \F_2) = 0
\]
as well.

It follows from the short exact sequence that $H^n(\vrm{S^1}{r};\Z) \otimes \F \cong H^n(\vrm{S^1}{r}; \F)$ for $\F=\Q$, $\R$, and $\F_2$.
Finally, the claim is obtained by naturality of the ring structure of cohomology with respect to the canonical coefficient morphism $\Z \longrightarrow \F$. 
\end{remark}

\begin{remark}[Naturality]
\label{rem: naturality}
If $k \ge 1$ and ${\rm incl} \colon \complex{1,3, \dots, 2k-1}^k \longrightarrow \complex{1,3, \dots, 2k+1}^{k+1}$ denotes the $S^1$-inclusion, we get a bundle morphism
\[
    \begin{tikzcd}
    ES^1\times_{S^1} S^{2k-1}_{1,3, \dots, 2k-1} \arrow[r, "{\rm incl}"] \arrow[d, "\pi"] & ES^1\times_{S^1} S^{2k+1}_{1,3, \dots, 2k+1} \arrow[d, "\pi"]\\
    BS^1 \arrow[r, equal] & BS^1
    \end{tikzcd}
\]
By naturality of the spectral sequences, we get that the induced map
\[
    {\rm incl}^* \colon H^*_{S^1}(S^{2k+1}_{1,3, \dots, 2k+1};\Z) \longrightarrow H^*_{S^1}(S^{2k-1}_{1,3, \dots, 2k-1};\Z) 
\]
is the canonical projection $\Z[u]/(1\cdot 3 \dots (2k+1)u^{k+1}) \longrightarrow \Z[u]/(1\cdot 3 \dots (2k-1)u^{k})$.
The analogous statement holds in cohomology with field coefficients $\F = \Q$, $\R$, or $\F_2$.
\end{remark}

\section{Equivariant homotopy equivalence $\vrm{S^1}{r}\simeq_{S^1} S^1_1*S^1_3*S^1_5*S^1_7*\ldots$}
\label{ssec:S1-he}

In this section, we identify a simpler (finite-dimensional) space that is $S^1$-homotopy equivalent to $\vrm{S^1}{r}$, which we use to show $\vrm{S^1}{r} \simeq_{S^1} S^1_1*S^1_3*\ldots*S^1_{2k+1}$ for $r \in [\frac{2\pi k}{2k+1}, \frac{2\pi(k+1)}{2k+3})$.
In~\cite{moyVRmS1}, a quotient map $q \colon \vrm{S^1}{r} \to \vrm{S^1}{r}/_{\sim}$ is constructed and shown to be a homotopy equivalence.
Briefly, the support of each measure $\mu \in \vrm{S^1}{r}$ can be partitioned into an odd number $n$ (depending on $\mu$) of clusters of close points.
An appropriately defined weighted average then sends $\mu$ to a measure supported on $n$ evenly spaced points on the circle, called a \emph{regular polygonal measure}, with all the mass in a cluster being moved to one of these points.
The quotient map $q$ then identifies measures that have the same average, and each class can be represented by this average regular polygonal measure.
The quotient $\vrm{S^1}{r}/_{\sim}$ is then shown to be a CW complex homotopy equivalent to $S^{2k+1}$ for $r \in [\frac{2\pi k}{2k+1}, \frac{2\pi(k+1)}{2k+3})$.
Our goal is to show that this can be upgraded to an $S^1$-homotopy equivalence, where $S^{2k+1}$ is viewed as $S^1_1*S^1_3*\ldots*S^1_{2k+1}$ with the diagonal action.
In Lemma~\ref{lem: homotopy equivalence}, we reprove the ordinary homotopy equivalence using a map that respects the symmetry.
In Theorem~\ref{thm:S^1-equivalence}, we show that this map is in fact an $S^1$-homotopy equivalence.
We introduce notation from~\cite{moyVRmS1} in Section~\ref{ssec:ehe-background}, construct the framework for Theorem~\ref{thm:S^1-equivalence} in Section~\ref{ssec:ehe-def-prop}, and complete the proof of this theorem in Section~\ref{ssec:proof-ehe}.

\subsection{Background and notation}\label{ssec:ehe-background}

We now recall notation from~\cite{moyVRmS1}.
For any delta measure $\delta_{\theta} \in \vrm{S^1}{r}$, let $E(\delta_\theta)$ be an open interval on $S^1$ of length $2\pi-2r$ centered at $e^{i(\theta+\pi)}$.
We can extend this definition to any measure $\mu \in \vrm{S^1}{r}$ by setting $E(\mu) = \bigcup_{e^{i \theta} \in \supp (\mu)} E(\delta_\theta)$.
This is the so-called {\em excluded region} of $\mu$.
Moreover, we set $A(\mu)$ to be the union of the connected components of $S^1 \setminus E(\mu)$ which contain at least one point of $\supp(\mu)$.
Now, $A(\mu)$ is a union of an odd number of (possibly singleton) connected closed intervals of $S^1$~\cite[Proposition~1]{moyVRmS1}, called {\em $\mu$-arcs}. 

For any $k \ge 0$, we let $V_{2k+1} \subseteq \vrm{S^1}{r}$ denote the set of measures $\mu$ for which $A(\mu)$ has exactly $2k+1$ connected components.
This defines a stratification $\vrm{S^1}{r} = \bigcup_{k \ge 0} V_{2k+1}$, where $V_{2k+1}$ is empty whenever $2k+1 > \frac{\pi}{\pi - r}$.
Even though $V_{2k+1}$ depends on the parameter $r$, we will omit $r$ from the notation for simplicity.
For more information and background, see~\cite[Section~3]{moyVRmS1}.\\

We can write any $\mu \in V_{2k+1}$ as $\mu = \sum_{j=0}^{2k} \sum_{i=0}^{n_j} a_{i,j} \delta_{\theta_{i,j}}$, where the $\mu$ arcs are ordered counterclockwise around the circle and each $\theta_{i,j}$ is in the $j^\text{th}$ $\mu$-arc.
We further choose a point $y$ in the excluded region of $\mu$ between the $2k^{\text{th}}$ and $0^{\text{th}}$ $\mu$-arcs and write each $\theta_{i,j}$ in the interval $(y,y+2\pi)$.
Then temporarily letting $m = \sum_{j=0}^{2k} \sum_{i=0}^{n_j} a_{i,j} (\theta_{i,j} - \frac{2 \pi j}{2k+1})$, we define a map $F_{2k+1} \colon V_{2k+1} \to V_{2k+1}$ by setting 
\[
    F_{2k+1}(\mu) = \sum_{j=0}^{2k} \left(\sum_{i=0}^{n_j} a_{i,j}\right) \delta_{m + \frac{2 \pi j}{2k+1}}.
\]
Thus, $F_{2k+1}(\mu)$ is a regular polygonal measure that should be understood as an appropriate weighted average of the delta masses of $\mu$.
We then define an equivalence relation $\sim$ on $\vrm{S^1}{r}$ by letting $\mu \sim \nu$ if and only if $\mu, \nu \in V_{2k+1}$ for some $k$ and $F_{2k+1}(\mu) = F_{2k+1}(\nu)$, and we then define
\begin{equation}
\label{eq: q}
    q \colon \vrm{S^1}{r} \longrightarrow \vrm{S^1}{r}/_\sim
    \quad\text{by}\quad
    \mu \longmapsto [F_{2k+1}(\mu)].
\end{equation}
The map $q$ is shown in~\cite[Theorem 1]{moyVRmS1} to be a homotopy equivalence.
Furthermore, the techniques used in the proof can all be carried out using $S^1$-equivariant maps.
Specifically, the action on $\vrm{S^1}{r}/_{\sim}$ is given by $\lambda_\theta \cdot q(\mu) = q(\lambda_\theta \cdot \mu)$, making $q$ equivariant; Propositions~2,3, and~4 of~\cite{moyVRmS1} all have equivariant versions; and the homotopies defined therein are all equivariant.
Thus, the map $q$ is in fact an $S^1$-homotopy equivalence, so we can reduce our goal of showing $\vrm{S^1}{r} \simeq_{S^1} S^1_1*S^1_3*\ldots*S^1_{2k+1}$ to showing $\vrm{S^1}{r}/_{\sim} \simeq_{S^1} S^1_1*S^1_3*\ldots*S^1_{2k+1}$.

For each integer $k \ge 0$, we define the following spaces of regular polygonal measures:
\begin{align*}
\overline{P}_{2k+1} &= \left\{\sum_{j=0}^{2k} a_j \delta_{\theta + \frac{2\pi j}{2k+1}}:~(a_j)_{j=0}^{2k} \in \Delta_{2k}, ~\theta \in \mathbb{R} \right\} \\
P_{2k+1} &= \left\{\sum_{j=0}^{2k} a_j \delta_{\theta + \frac{2\pi j}{2k+1}} \in \overline{P}_{2k+1}:~(a_j)_{j=0}^{2k} \in \relint(\Delta_{2k}) \right\}.
\end{align*}
Here $\Delta_{2k}=\{(a_j)_{j=0}^{2k} :~ a_j \ge 0\text{ for all $j$ and } \sum_{j=0}^{2k}a_j=1\}$ is the $2k$-dimensional simplex, and $\relint(\Delta_{2k}) = \{(a_j)_{j=0}^{2k} \in \Delta_{2k}:~ a_j > 0 \text{ for all $j$}\}$ denotes the {\em relative interior} of the simplex $\Delta_{2k}$.\\

The quotient map $q$ identifies each measure of $\vrm{S^1}{r}$ with a measure in some $P_{2k+1}$, so we have $\vrm{S^1}{r}/_{\sim}\, = q(P_1) \cup q(P_3) \cup \dots \cup q(P_{2K+1})$, where $K = \lfloor \frac{r}{2 \pi - 2r} \rfloor$ is the largest integer $k$ such that $\vrm{S^1}{r}$ contains a measure of $P_{2k+1}$.
We set 
\[
X_{2k+1} \coloneqq q(P_1) \cup q(P_3) \cup \dots \cup q(P_{2k+1}) \subseteq \vrm{S^1}{r}/_{\sim}
\]
for each $k \ge 0$.
Note that $X_{2k+1}$ is the $(2k+1)$-skeleton of the CW complex of~\cite{moyVRmS1}, known to be homotopy equivalent to $S^{2k+1}$.\\

The centers of the regular polygons of $P_{2k+1}$ form a circle.
We will use subscripts to keep track of these circles, letting
\[
S^1_{2k+1} = \left\{ \sum_{j=0}^{2k} \tfrac{1}{2k+1} \delta_{\theta + \frac{2 \pi j}{2k+1}} \colon~ \theta \in \R \right\}.
\]
Each $S^1_{2k+1}$ inherits an action of the circle group from the action on $\vrm{S^1}{r}$, so the action on $S^1_{2k+1}$ has kernel generated by $\lambda_{\frac{2 \pi}{2k+1}}$.
Note that this agrees with the definition of $S^1_{2k+1}$ used in Section~\ref{sec:epc-S1join} after identifying each $S^1_{2k+1}$ with the usual circle $S^1$.

As in Section~\ref{sec:epc-S1join}, we also define a homeomorphic copy of the $(2k+1)$-sphere by 
\[
S^{2k+1}_{2k+1} = S^1_{1} \ast S^1_{3} \ast \dots \ast S^1_{2k+1}
\]
with the diagonal action of the circle group.
We write elements of $S^{2k+1}_{2k+1}$ as formal linear combinations.

It will be useful for later to describe a sum-decomposition of elements of $\overline{P}_{2k+1}$.
Namely, given $\mu=\sum_{j=0}^{2k} a_j \delta_{\theta + \frac{2\pi j}{2k+1}} \in \overline{P}_{2k+1}$, let 
\[
    t:= (2k+1)\cdot\min_{0 \le j \le 2k}\{a_j\} \in [0,1].
\]
Then $\mu$ can be written as the convex combination
\begin{equation}
\label{eq: decomposition of P_bar}
\mu=(1-t)\nu + t\sum_{j=0}^{2k}\tfrac{1}{2k+1}\delta_{\theta + \frac{2\pi j}{2k+1}}
\end{equation}
for $\nu \in \partial P_{2k+1} := \overline{P}_{2k+1} \setminus P_{2k+1}$, and we will soon define our homotopy equivalence to send $\mu$ to a linear combination in $S^{2k+1}_{2k+1}$ with its last component equal to $t\sum_{j=0}^{2k}\frac{1}{2k+1}\delta_{\theta + \frac{2\pi j}{2k+1}}$.

We will also need the homeomorphism types of $P_{2k+1}$ and $\overline{P}_{2k+1}$.
Let $\sim$ be the equivalence relation on $\Delta^{2k} \times I$, where $I\coloneqq [0, \tfrac{2\pi}{2k+1}]$, generated by 
\[
\big((a_0, a_1, \dots, a_{2k}),~ 0\big) \sim \big((a_1, \dots, a_{2k}, a_0),~  \tfrac{2\pi}{2k+1}\big),
\]
for $(a_0, \dots, a_{2k}) \in \Delta^{2k}$.
Then the map
\begin{equation*}
(\Delta^{2k} \times I)/_{\sim} \longrightarrow \overline{P}_{2k+1}, \hspace{5mm} [(a_0, \dots, a_{2k}), \theta] \longmapsto \sum_{j=0}^{2k} a_j \delta_{\theta+ \frac{2\pi j}{2k+1}}
\end{equation*}
is a homeomorphism and restricts to a homeomorphism on the boundary:
\begin{equation}
\label{eq: mapping torus}
\begin{tikzcd}
(\Delta^{2k} \times I)/_{\sim} \arrow[r, "\approx"] \arrow[d, hookleftarrow] & \overline{P}_{2k+1} \arrow[d, hookleftarrow]\\
(\partial\Delta^{2k} \times I)/_{\sim}  \arrow[r, "\approx"] & \partial P_{2k+1}.
\end{tikzcd}
\end{equation}
The space $(\Delta^{2k}\times I)/_{\sim}$ is the mapping torus of a simplex automorphism
\[
    \Delta^{2k} \longrightarrow \Delta^{2k}, ~ (a_0, a_1, \dots, a_{2k}) \longmapsto (a_1, \dots, a_{2k}, a_0),
\]
which preserves orientation, since a cycle of an odd number of elements is an even permutation.
Hence, $(\Delta^{2k}\times I)/_{\sim}$ is homeomorphic to $D^{2k} \times S^1$; see Figure~\ref{fig:S1_action_on_P3}.
The above simplex automorphism restricts to an automorphism of $\partial \Delta_{2k}$ of degree $(-1)^{2k}=1$, so its mapping torus $(\partial\Delta^{2k}\times I)/_{\sim}$ is homeomorphic to the product $S^{2k-1} \times S^1$: 
\begin{equation*}
\label{eq:mapping torus is trivial}
\begin{tikzcd}
		D^{2k} \times S^1 \arrow[r, "\approx"] \arrow[d, hookleftarrow] & (\Delta^{2k} \times I)/_{\sim} \arrow[d, hookleftarrow]\\
		S^{2k-1} \times S^1  \arrow[r, "\approx"] & (\partial\Delta^{2k} \times I)/_{\sim}.
\end{tikzcd}
\end{equation*}
Composing the previous two square diagrams we get
\begin{equation}
\label{eq: bdr P into over P is cofibration}
	\begin{tikzcd}
		D^{2k} \times S^1 \arrow[r, "\approx"] \arrow[d, hookleftarrow] & \overline{P}_{2k+1} \arrow[d, hookleftarrow]\\
		S^{2k-1} \times S^1  \arrow[r, "\approx"] & \partial P_{2k+1}.
\end{tikzcd}
\end{equation}

\begin{figure}[h]
    \centering
    \includegraphics[width=0.4\textwidth]{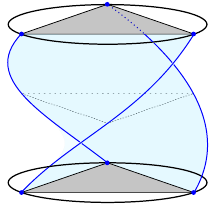}
    \caption{The mapping torus identification of $\overline{P}_{2k+1}$~\eqref{eq: mapping torus} when $2k+1=3$.
    The top and bottom faces are identified.
    From bottom to top, the angle $\theta$ varies over $[0,\frac{2\pi}{3}]$.
    The three blue curves trace the movement of a single vertex under the action of $S^1$.}
    \label{fig:S1_action_on_P3}
\end{figure}

We also define
\[
R_{2k} \coloneqq \left\{\sum_{j=0}^{2k}a_j \delta_{\frac{2\pi j}{2k+1}}:~(a_j)_{j=0}^{2k} \in \Delta_{2k} \right\} \cong \Delta_{2k}
\]
and 
\[
\partial R_{2k} \coloneqq \left\{\sum_{j=0}^{2k}a_j \delta_{\frac{2\pi j}{2k+1}}:~(a_j)_{j=0}^{2k} \in \partial \Delta_{2k} \right\} \cong \partial\Delta_{2k}.
\]
The approach of~\cite{moyVRmS1} uses $q(R_{2k})$ as a $2k$-cell of a $CW$ complex and shows the $2k$-skeleton is contractible; the remaining points of $q(P_{2k+1})$ then form the single $(2k+1)$-cell of the skeleton $X_{2k+1}$, and contracting the $2k$-skeleton to a point shows $X_{2k+1} \simeq S^{2k+1}$.
However, this approach ignores the symmetry of the circle, treating the measures of $R_{2k}$ differently than rotated copies.
In the following section, we describe directly how the spaces $q(P_{2k+1})$ glue together, giving a homotopy equivalence that respects the symmetry of the circle.

\subsection{Definition of $\varphi_{2k+1}$ and first properties}
\label{ssec:ehe-def-prop}

Recall $X_{2k+1} = q(P_1) \cup q(P_3) \cup \dots \cup q(P_{2k+1})$ and $S^{2k+1}_{2k+1} = S^1_{1} \ast S^1_{3} \ast \dots \ast S^1_{2k+1}$.
We recursively define maps 
\[
\varphi_{2k+1}: X_{2k+1} \longrightarrow S^{2k+1}_{2k+1}
\]
for $k \ge 0$, such that each is $S^1$-equivariant and extends the previously defined $\varphi_{2k-1}$ if $k \ge 1$; our eventual goal is to show each $\varphi_{2k+1}$ is an equivariant homotopy equivalence.
For $k=0$, we let $\varphi_1 \colon X_1 \to S^1_1$ send a class to the unique delta measure representing it: $\varphi_1([\delta_{\theta}]) = \delta_{\theta}$.

Let $k \ge 1$ and suppose an $S^1$-equivariant map $\varphi_{2k-1} \colon X_{2k-1} \to S^{2k-1}_{2k-1}$ has already been defined.
First define maps
\begin{align}
\label{eq:map-alpha}
\alpha_{2k-1} ~:~ \partial P_{2k+1} ~~&\longrightarrow ~~ S^{2k-1}_{2k-1} \times S^1_{2k+1}\\ 
\sum_{j=0}^{2k}a_j \delta_{\theta + \frac{2\pi j}{2k+1}} &\longmapsto \Bigg((\varphi_{2k-1} \circ q)\Big(\sum_{j=0}^{2k}a_j \delta_{\theta + \frac{2\pi j}{2k+1}}\Big),~ \sum_{j=0}^{2k} \tfrac{1}{2k+1} \delta_{\theta + \frac{2\pi j}{2k+1}}\Bigg) \notag
\end{align}
and
\begin{align}
\label{eq:map-beta}
\beta_{2k+1}: \overline{P}_{2k+1} &\longrightarrow C(S^{2k-1}_{2k-1}) \times S^1_{2k+1}\\ 
(1-t)\nu + t\sum_{j=0}^{2k} \tfrac{1}{2k+1} \delta_{\theta + \frac{2\pi j}{2k+1}} &\longmapsto \Bigg((1-t)(\varphi_{2k-1} \circ q)(\nu)+tu,~ \sum_{j=0}^{2k} \tfrac{1}{2k+1} \delta_{\theta + \frac{2\pi j}{2k+1}}\Bigg), \notag
\end{align}
where $t \in [0,1]$ is unique such that $\nu \in \partial P_{2k+1}$ (see~\eqref{eq: decomposition of P_bar}) and the cone $C(S^{2k-1}_{2k-1})$ is identified with $S^{2k-1}_{2k-1}*u$.

In the following commutative diagram~\eqref{eq:cube}, the front and back squares are pushouts: for the back square, this follows from the CW structure for $X_{2k+1}$ in~\cite{moyVRmS1}, while the front is the product-into-join pushout.
We define $\varphi_{2k+1}$ to be the induced map on pushouts, represented in the diagram by the dashed arrow.
We will show below that all the diagonal maps are homotopy equivalences.
\begin{equation}
\label{eq:cube}
\begin{tikzcd}
\partial P_{2k+1} \arrow[dd] \arrow[rr, hook, "i"] \arrow[rd, "\alpha_{2k-1}"] & & \overline{P}_{2k+1} \arrow[dd] \arrow[rd, "\beta_{2k+1}"] &                                                 \\
& S^{2k-1}_{2k-1} \times S^1_{2k+1} \arrow[rr, hook, crossing over, "j \hspace{.5cm}"] & & C(S^{2k-1}_{2k-1}) \times S^1_{2k+1} \arrow[dd] \\
X_{2k-1} \arrow[rr, hook] \arrow[rd, "\varphi_{2k-1}"] & & X_{2k+1} \arrow[rd, "\varphi_{2k+1}", dashed] & \\
& S^{2k-1}_{2k-1} \arrow[from=uu, crossing over] \arrow[rr, hook] & & S^{2k+1}_{2k+1} 
\end{tikzcd}
\end{equation}

An explicit formula for $\varphi_{2k+1}$ is
\begin{align}
\label{eq: map phi}
\varphi_{2k+1} \colon X_{2k+1} &\longrightarrow S^{2k-1}_{2k-1} * S^1_{2k+1}\\ 
\big[ (1-t_{2k+1}) \nu_{2k+1} + t_{2k+1} \mu_{2k+1} \big] &\longmapsto (1-t_{2k+1})(\varphi_{2k-1} \circ q)(\nu_{2k+1})+t_{2k+1} \mu_{2k+1}, \notag
\end{align}
where we have chosen a regular polygonal representative 
\[
(1-t_{2k+1}) \nu_{2k+1} + t_{2k+1} \mu_{2k+1} \in \overline{P}_{2k+1},
\]
with $\mu_{2k+1} = \sum_{j=0}^{2k} \frac{1}{2k+1} \delta_{\theta + \frac{2\pi j}{2k+1}} \in S^1_{2k+1}$ for some $\theta$, and again $t_{2k+1} \in [0,1]$ is unique such that $\nu_{2k+1} \in \partial P_{2k+1}$.
The convex combination on the right hand side represents the coordinates in the join $S^{2k-1}_{2k-1} * S^1_{2k+1} = S^{2k+1}_{2k+1}$.
Applying this formula recursively sends a class $[\mu] \in X_{2k+1}$ to a linear combination $t_1 \mu_1 + t_3 \mu_3 + \dots + t_{2k+1} \mu_{2k+1} \in S^{2k+1}_{2k+1}$.
This is analogous to the Fourier series for an odd function, since each $t_{2l+1}$ can be understood as describing the amount of the $(2l+1)$-gonal measure $\mu_{2l+1}$ required to build $\mu$.\\

Our main result of this section is the following.

\begin{theorem}
\label{thm:S^1-equivalence}
Let $k \ge 0$ be an integer and $\frac{2\pi k}{2k+1}\le r<\frac{2\pi(k+1)}{2k+3}$.
Then the composition 
\[
    \vrm{S^1}{r} \xrightarrow{~~~~q~~~~~~~} X_{2k+1} \xrightarrow{~\varphi_{2k+1}~} S^{2k+1}_{2k+1}
\]
is an $S^1$-homotopy equivalence. 
Moreover, if $\frac{2\pi (k+1)}{2k+3}\le r'<\frac{2\pi(k+2)}{2k+5}$, then the diagram
\begin{equation} \label{eq: functoriality of S^1-homotopy equivalence}
    \begin{tikzcd}
    \vrm{S^1}{r} \arrow[r, "q"] \arrow[d, hook] & X_{2k+1} \arrow[r, "\varphi_{2k+1}"] \arrow[d, hook] & S^{2k+1}_{2k+1} \arrow[d, hook] \\
    \vrm{S^1}{r'} \arrow[r, "q"] & X_{2k+3} \arrow[r, "\varphi_{2k+3}"] & S^{2k+3}_{2k+3}
    \end{tikzcd}
\end{equation}
commutes.
\end{theorem}

It was already observed in Section~\ref{ssec:ehe-background} that $q$ is an $S^1$-homotopy equivalence. Therefore, for the first part of the theorem, it is enough to show that $\varphi_{2k+1}$ is.
This will be the main focus for the rest of the section.
As for the functoriality part of the theorem,
the left square in \eqref{eq: functoriality of S^1-homotopy equivalence} commutes due to the definition \eqref{eq: q} of $q$, while the right square commutes by the inductive formula \eqref{eq: map phi} of $\varphi_{2k+3}$.\\

For a $G$-space $X$ we will denote by $X^G = \{x \in X \colon~ gx=x \textrm{ for every } g\in G\}$ the \emph{fix-point set} of $G$.
To prove that $\varphi_{2k+1}$ is an $S^1$-homotopy equivalence, we will use the following lemma.

\begin{lemma}
\label{lem: S^1-equivalence}
Suppose $X$ and $Y$ are $S^1$-spaces such that for any closed subgroup $H \subseteq S^1$, the fix-point sets $X^H$ and $Y^H$ are finite CW complexes.
Then an $S^1$-equivariant map $f\colon X \to Y$ is an $S^1$-homotopy equivalence if and only if for any closed subgroup $H \subseteq S^1$, the induced map $f^H\colon X^H \to Y^H$ is a homotopy equivalence.
\end{lemma}

\begin{proof}
Due to~\cite[Proposition~8.2.6]{tom2006transandrepres} (see also~\cite[Theorem~1.1]{james1978equivariant}), the equivalence holds if $X$ and $Y$ are $S^1$-ANR's (absolute neighborhood retracts).
Let us show that our assumptions imply that $X$ and $Y$ are $S^1$-ANR's.
Due to symmetry, we will consider only the space $X$.
	
Since the fix-point sets $X^H$ are finite CW complexes, they are also ENR's (Euclidean neighborhood retracts)~\cite[Corollary~A.10]{Hatcher}.
Since the class of ENR's coincides with the class of locally compact, separable, and finite-dimensional ANR's~\cite[Section 5.3]{Srzednicki04}, we get that $X^H$ is an ANR for every closed subgroup $H\subseteq S^1$.
This is equivalent to $X$ being an $S^1$-ANR, due to~\cite[Proposition~5.2.6]{tom2006transandrepres}.
\end{proof}

Therefore, in order to prove Theorem~\ref{thm:S^1-equivalence}, it is enough to show that for every closed subgroup $H \subseteq S^1$,
\begin{enumerate} [(i)] 
\item \label{enum: fix-point-sets are cw} the fix-point sets $(X_{2k+1})^H$ and $(S^{2k+1}_{2k+1})^H$ are finite CW-complexes, and
\item \label{enum: phi^H is equivalence} the induced map $(\varphi_{2k+1})^H: (X_{2k+1} )^H \to (S^{2k+1}_{2k+1})^H$ is a homotopy equivalence.
\end{enumerate}

The proper closed subgroups of $S^1$ are $\Z_n \coloneqq \Z/n$, for integers $n \ge 2$.
Conditions~\eqref{enum: fix-point-sets are cw} and \eqref{enum: phi^H is equivalence} hold for $H=S^1$, since $S^1$ acts fix-point free on all spaces involved. 
On the other hand, the case when $H$ is the trivial group is treated below in Lemma~\ref{lem: homotopy equivalence}, to which we will reduce the more general case $H=\Z_n$ in the next section. 
Moreover, the lemma reproves the homotopy equivalence of~\cite[Theorem~2]{moyVRmS1}, which is the non-equivariant version of Theorem~\ref{thm:S^1-equivalence}.

\begin{lemma}
\label{lem: homotopy equivalence}
For any integer $k \ge 0$, the map $\varphi_{2k+1}: X_{2k+1} \longrightarrow S^{2k+1}_{2k+1}$ is a homotopy equivalence.
\end{lemma}
\begin{proof}
The proof follows by induction on $k \ge 0$.
For the base case $k = 0$, we notice that $\varphi_1$ is a homeomorphism.
	
Let $k \ge 1$ and assume $\varphi_{2k-1}:X_{2k-1} \to S^{2k-1}_{2k-1}$ is a homotopy equivalence.
To be able to apply the gluing theorem for adjunction spaces~\cite[Theorem~7.5.7]{brown2006topology}, we need to show that in~\eqref{eq:cube}, the maps $i$ and $j$ are closed cofibrations and that the maps $\alpha_{2k-1}$ and $\beta_{2k+1}$ are homotopy equivalences.
	
The map $j$ is a closed cofibration since it is the inclusion of a CW subcomplex.
The map $i$ is a closed cofibration since it fits into the square diagram~\eqref{eq: bdr P into over P is cofibration} as the right vertical map where the left vertical map is an inclusion of a CW subcomplex.
	
Let us show that $\alpha_{2k-1}$, defined in \eqref{eq:map-alpha}, is a homotopy equivalence.
Since $\partial P_{2k+1}$ is a mapping torus homeomorphic to the product (see~\eqref{eq: mapping torus} and~\eqref{eq: bdr P into over P is cofibration}), we know that
\[
\partial P_{2k+1} \longrightarrow S^1_{2k+1},\quad
\sum_{j=0}^{2k}a_j\delta_{\theta+\frac{2\pi j}{2k+1}} \longmapsto \sum_{j=0}^{2k}\tfrac{1}{2k+1}\delta_{\theta+\frac{2\pi j}{2k+1}}
\]
is a trivial fiber bundle with the fiber over the base point $\sum_{j=0}^{2k}\frac{1}{2k+1}\delta_{\frac{2\pi j}{2k+1}}$ equal to $\partial R_{2k} \cong \partial \Delta_{2k}$.
Thus, $\alpha_{2k-1}$ is a fiber bundle map over the identity between trivial fibrations:
\begin{equation*}
\begin{tikzcd}
	\partial R_{2k}  \arrow[rr, "\varphi_{2k-1}~\circ~ q"] && S^{2k-1}_{2k-1}\\
	\partial P_{2k+1}  \arrow[rr, "\alpha_{2k-1}"] \arrow[d, twoheadrightarrow]\arrow[u, hookleftarrow] && S^{2k-1}_{2k-1}\times S^1_{2k+1} \arrow[d, twoheadrightarrow, "{\rm pr}_2"] \arrow[u, hookleftarrow] \\
	S^1_{2k+1} \arrow[rr, "\id"] && S^1_{2k+1}.
\end{tikzcd}
\end{equation*}
The map between fibers is a homotopy equivalence because $\varphi_{2k-1}\colon X_{2k-1} \to S^{2k-1}_{2k-1}$ is a homotopy equivalence due to the inductive hypothesis and $q\colon \partial R_{2k} \to X_{2k-1}$ is a homotopy equivalence due to~\cite[Lemma~14]{moyVRmS1}.
Thus, we conclude that $\alpha_{2k-1}$ is a homotopy equivalence from the 5-lemma applied to the morphism between long exact sequences of fibrations and Whitehead's theorem.

An analogous argument shows that the map $\beta_{2k+1}$, defined in \eqref{eq:map-beta}, is a homotopy equivalence.
This completes the inductive step as explained at the beginning of the proof.
\end{proof}

\subsection{Proof of equivariant homotopy equivalence}
\label{ssec:proof-ehe}

We now prove Theorem~\ref{thm:S^1-equivalence} by verifying conditions \eqref{enum: fix-point-sets are cw} and \eqref{enum: phi^H is equivalence} above, namely, that $(X_{2k+1})^H$ and $(S^{2k+1}_{2k+1})^H$ are finite CW-complexes, and
that $(\varphi_{2k+1})^H: (X_{2k+1} )^H \to (S^{2k+1}_{2k+1})^H$ is a homotopy equivalence.

Let us first explain why it is enough to consider only $H = \Z_n$ for odd $n \ge 1$ dividing $2k+1$.
We have
\[
(\overline{P}_{2k+1})^{\Z_n} =
\begin{cases}
\left\{\sum_{j=0}^{2k} a_j \delta_{\theta+ \frac{2\pi j}{2k+1}} \in \overline{P}_{2k+1}:~ a_j = a_{j+m ~{\rm mod}~2k+1}\right\} & 2k+1 = n m~ {\rm for } ~m\in \Z_{>0} \\
\emptyset & {\rm otherwise.}
\end{cases}
\]
In particular, from the back push-out square in \eqref{eq:cube} and by induction on $k \ge 0$, it follows that
\begin{align*}
X_{2k+1}^{\Z_n} & \approx 
\begin{cases}
    X_{mn}^{\Z_n} &n,~m{\rm ~odd~and~ } ~mn \le 2k+1 < (m+2)n\\
    \emptyset & {\rm otherwise.}
\end{cases}
\end{align*}
Similarly,
\begin{align*}
(S^{2k+1}_{2k+1})^{\Z_n} & \approx 
\begin{cases}
    S^1_n * S^1_{3n} * \dots * S^1_{mn} &n,~m{\rm ~odd~and~ } ~mn \le 2k+1 < (m+2)n\\
    \emptyset & {\rm otherwise.}
\end{cases}
\end{align*}
Therefore, without loss of generality, we may assume $H = \Z_n \subseteq S^1$ with $n$ odd and $2k+1=nm$, for some odd integer $m\ge 1$.

\begin{proof}[Proof of Theorem~\ref{thm:S^1-equivalence}]
As it was already noted, it is enough to show that $\varphi_{2k+1} \colon X_{2k+1} \longrightarrow S^{2k+1}_{2k+1}$ is an $S^1$-homotopy equivalence for every $k \ge 0$.
For this, it suffices to verify~(\ref{enum: fix-point-sets are cw}) and~(\ref{enum: phi^H is equivalence}) of the previous section. 

As explained before the proof, we may assume $H=\Z_n$ where $n \ge 1$ is odd and $2k+1=nm$ for some odd $m \ge 1$.
Define a map
\begin{align*}
    g_{n} \colon \vrm{S^1}{\pi-s} &\longrightarrow \vrm{S^1}{\pi-\tfrac{s}{n}} \\
    \mu = \sum_{j=0}^l a_j \delta_{\theta_j} &\longmapsto \sum_{t=0}^{n-1} \sum_{j=0}^l \tfrac{a_j}{n} \delta_{\frac{\theta_j}{n}+\frac{2 \pi t}{n}}
\end{align*}
with $a_j>0$, where we choose representative angles $\theta_j \in [0, 2 \pi)$ for each $j$.
See Figure~\ref{fig:g_n map} for an illustration of this map.
We will continue to identify points on the circle by angles, where angles are written modulo $2 \pi$.

\begin{figure}
    \centering
    \includegraphics[width=0.8\linewidth]{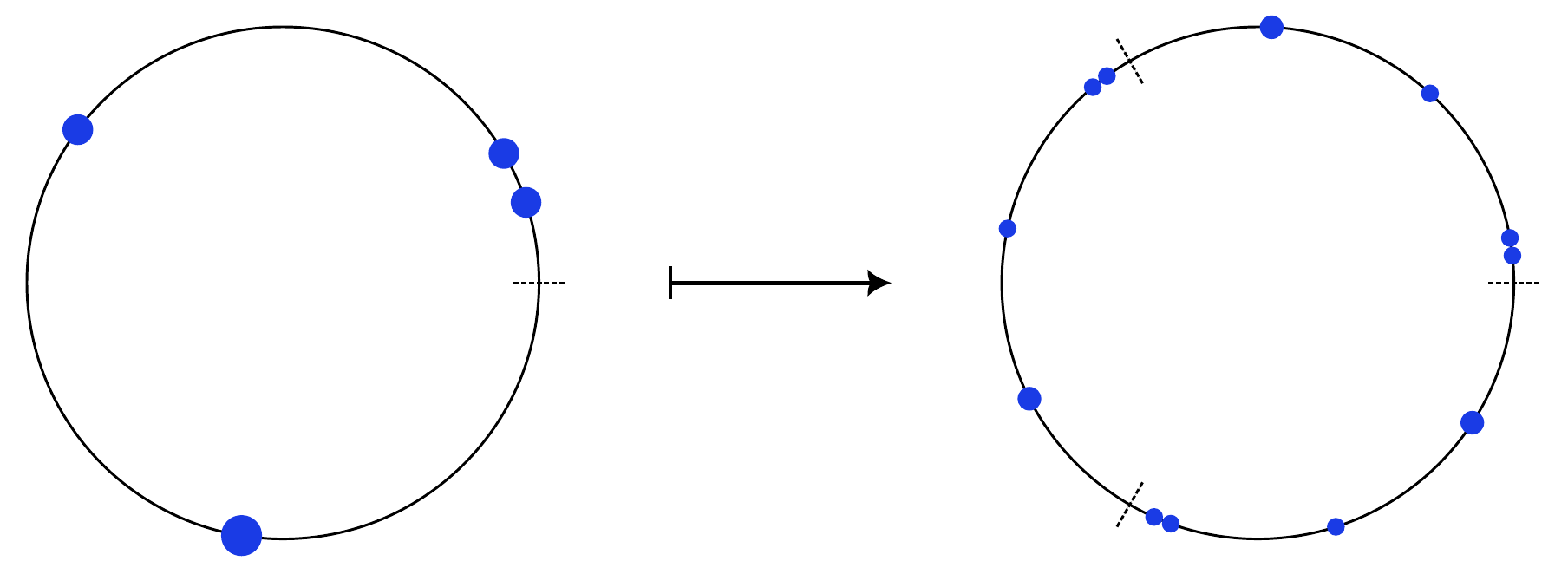}
    \caption{An example of $g_n$ with $n=3$.
    A measure $\mu$ in the domain is shown on the left, with four points in $\supp(\mu)$ and three $\mu$-arcs.
    Its image $g_3(\mu)$, on the right, has twelve points in its support and nine arcs.}
    \label{fig:g_n map}
\end{figure}

To see that $g_n$ is well-defined, for a given $\theta_j \in \supp(\mu)$, write a point of $\supp(\mu)$ farthest from $\theta_j$, and thus nearest to $\theta_j+\pi$, as $\theta_j + \pi+x$ with $s \leq |x| \leq \pi$.
Then for each $\tau = \frac{\theta_j}{n} + \frac{2 \pi t}{n} \in \supp(g_n(\mu))$, because $n$ is odd, a point in $\supp(g_n(\mu))$ nearest to the antipodal point 
\[
    \tau + \pi = \tfrac{\theta_j}{n} + \tfrac{2 \pi t}{n} + \pi = \tfrac{\theta_j+\pi}{n} + \tfrac{2 \pi \left(\frac{n-1}{2}+t\right)}{n}
\]
must be $\frac{\theta_j+\pi+x}{n} + \frac{2 \pi \left(\frac{n-1}{2}+t\right)}{n}$.
This support point has a distance of at least $\frac{s}{n}$ from $\tau + \pi$ and thus has a distance of at most $\pi - \frac{s}{n}$ from $\tau$, as required.

The map $g_n$ can be understood as a ``pullback'' map of the $n$-sheeted cover $f_n \colon S^1 \to S^1$ given by $f_n(e^{i\theta}) = e^{n i \theta}$, in the sense that for any small enough (diameter less than $\frac{2\pi}{n}$) measurable set $U \subseteq S^1$, we have $g_n(\mu)(U) = \frac{1}{n}\mu(f_n(U))$.

We first check that $g_n$ respects the equivalence relations from Section~\ref{ssec:ehe-background}, originally defined in~\cite{moyVRmS1}.
Each point $\theta_j \in \supp(\mu)$ excludes an arc $(\theta_j + \pi -s, \theta_j + \pi +s)$ (see Section~\ref{ssec:ehe-background}), and the corresponding $n$ points $\frac{\theta_j}{n} + \frac{2 \pi t}{n} \in \supp(g_n(\mu))$ with $0\leq t \leq n-1$ exclude the arcs 
\[
    \left(\tfrac{\theta_j-s}{n} + \pi + \tfrac{2 \pi t}{n},\tfrac{\theta_j+s}{n} + \pi + \tfrac{2 \pi t}{n}\right).
\]
Equivalently, by reindexing, we find that the $(l+1)n$ points of $\supp(g_n(\mu))$ exclude arcs 
\[
    \left(\tfrac{\theta_j + \pi -s}{n} + \tfrac{2 \pi t}{n},\tfrac{\theta_j + \pi +s}{n} + \tfrac{2 \pi t}{n}\right)
\]
for all $t$ and $j$.
Temporarily choosing an angle $y_0$ in the excluded region of $\mu$ and writing $\frac{1}{n}U = \{\frac{\theta}{n} : \theta \in U, y_0 < \theta < y_0 + 2 \pi\}$ for any $U \subset S^1$ not containing $y_0$, the above computation shows that for each $\mu$-arc $A$, we get $g_n(\mu)$-arcs $\frac{1}{n}A + \frac{2 \pi t}{n}$ for each $t$.
These arcs further have the property that for a measurable set $U \subseteq A$, we have $g_n(\mu)(\frac{1}{n}U + \frac{2 \pi t}{n}) = \frac{1}{n}\mu(U)$.
Since the quotient map $q$ is defined in terms of weighted averages \eqref{eq: q}, it follows that if $q(\mu_1) = q(\mu_2)$, then $q(g_n(\mu_1)) = q(g_n(\mu_2))$.
Thus, $g_n$ descends to a map $q(\vrm{S^1}{\pi-s}) \to q(\vrm{S^1}{\pi-\tfrac{s}{n}})$ on quotient spaces.
Restricting to skeletons gives maps
\[
    \overline{g}_{m,n} \colon X_m \longrightarrow X_{mn}.
\]
\fussy

The map $g_n$ has image equal to the fix-point set $\big( \vrm{S^1}{\pi-\tfrac{s}{n}} \big)^{\mathbb{Z}_n}$.
Furthermore, every class in $X_{mn}^{\mathbb{Z}_n}$ can be represented by a regular polygonal measure, where necessarily $n$ divides the number of vertices and the masses are periodic, so we also find that $\overline{g}_{m,n}$ has image equal to the fix-point set $X_{mn}^{\mathbb{Z}_n}$.
Since each class in $X_m$ corresponds to a unique regular polygonal measure, and distinct regular polygonal measures are sent to distinct measures by $\overline{g}_{m,n}$, we find that $\overline{g}_{m,n}$ is injective.
Furthermore, since the skeletons are compact and Hausdorff, $\overline{g}_{m,n}$ is a homeomorphism onto its image, so its restriction gives $X_m \cong X_{mn}^{\mathbb{Z}_n}$.
As $X_m$ has the CW structure described in~\cite{moyVRmS1}, we have shown that $X_{mn}^{\mathbb{Z}_n}$ is a finite CW complex, as required.

Next, define $h_{m,n} \colon S^m_m \to S^{mn}_{mn}$ as the composition
\[
    \begin{tikzcd}
        S^1_1 \ast S^1_3 \ast \cdots \ast S^1_{m-2} \ast S^1_m \arrow[r] & S^1_{n} \ast S^1_{3n} \ast \cdots \ast S^1_{(m-2)n} \ast S^1_{mn} \arrow[r, hook] & S^1_1 \ast S^1_3 \ast \cdots \ast S^1_{mn-2} \ast S^1_{mn}.
    \end{tikzcd}
\]
Here, the first map above is the join of the restrictions $g_n|_{S^1_k} \colon S^1_{k} \xrightarrow{~~\approx~~} S^1_{kn}$ for each odd $k$, each of which is a homeomorphism.
The second map above is obtained by inserting zero coefficients in the entries where the subscript is not divisible by $n$.
Expressing elements of the join as formal linear combinations, we have an explicit formula
\[
h_{m,n}(t_1\mu_1 + t_3 \mu_3 + \dots + t_m \mu_m) = t_1 g_n(\mu_1) + t_3 g_n(\mu_3) + \dots + t_m g_n (\mu_m).
\]
The map $h_{m,n}$ is thus a homeomorphism onto its image $S^1_{n} \ast S^1_{3n} \ast \cdots \ast S^1_{(m-2)n} \ast S^1_{mn}$, which is exactly $(S^{mn}_{mn})^{\mathbb{Z}_n}$.
Since $(S^{mn}_{mn})^{\mathbb{Z}_n} \cong S^m_m$ is a finite CW complex, we have accomplished our first goal of showing~(\ref{enum: fix-point-sets are cw}) of Section~\ref{ssec:ehe-def-prop}.

We next show that the diagram 
\[
\begin{tikzcd}
X_m \arrow[rr, "{\overline{g}_{m,n}}", "\approx"'] \arrow[d, "\varphi_m"', "\simeq"] && X_{mn} \arrow[d, "\varphi_{mn}"] \\
S^m_m \arrow[rr, "{h_{m,n}}", "\approx"']                                 && S^{mn}_{mn}                
\end{tikzcd}
\]
commutes.
Since $\varphi_m$ is a homotopy equivalence, commutativity of the diagram will show $\varphi_{mn}^{\mathbb{Z}_n}$ is a homotopy equivalence, giving~(\ref{enum: phi^H is equivalence}) of Section~\ref{ssec:ehe-def-prop}.

We check commutativity via the explicit formula for $\varphi_{2k+1}$ in~\eqref{eq: map phi}, using induction on $m$.
For the base case, if $m=1$, then 
\[\varphi_n(\overline{g}_{1,n}([\delta_{\theta}])) 
= \varphi_n\left(\left[\sum_{j=0}^{n-1} \tfrac{1}{n} \delta_{\frac{\theta}{n} + \frac{2 \pi j}{n}}\right]\right) 
= \sum_{j=0}^{n-1} \tfrac{1}{n} \delta_{\frac{\theta}{n} + \frac{2 \pi j}{n}} 
=g_n(\delta_\theta)
= h_{1,n}(\delta_{\theta}) 
= h_{1,n}(\varphi_1([\delta_{\theta}])),
\]
where the middle term is an element of $S^n_n = S^1_1 \ast S^1_3 \ast \cdots \ast S^1_n$ contained in the copy of $S^1_n$.
This confirms the base case.

For the inductive step, we prove the diagram above commutes for a given odd $m$ assuming $\varphi_{(m-2)n} \circ \overline{g}_{m-2,n} = h_{m-2,n} \circ \varphi_{m-2}$.
By~\eqref{eq: decomposition of P_bar}, we can express any class in $X_m$ as $[\mu]$ with $\mu = (1-t_m)\nu_m + t_m \mu_m$, where $\nu_m \in \partial P_{m}$ and $\mu_m = \sum_{j=0}^{m-1} \frac{1}{m} \delta_{\theta + \frac{2 \pi j}{m}}$ for some $\theta$.
Then using~\eqref{eq: map phi}, we have 
\begin{align*}
h_{m,n}(\varphi_m([\mu])) 
&= h_{m,n}\big((1-t_m)\varphi_{m-2}([\nu_m]) + t_m \mu_m\big)\\
&= (1-t_m)h_{m-2,n}(\varphi_{m-2}([\nu_m])) + t_m g_n(\mu_m)\\
&= (1-t_m)\varphi_{(m-2)n}(\overline{g}_{m-2,n}([\nu_m])) + t_m g_n(\mu_m).
\end{align*}
On the other hand,
\begin{align*}
\varphi_{mn}(\overline{g}_{m,n}([\mu])) 
&= \varphi_{mn}([g_n(\mu)]) \\
&= \varphi_{mn}\Big(\big[g_n((1-t_m)\nu_m+t_m\mu_m)\big]\Big) \\
&= \varphi_{mn}\Big(\big[(1-t_m)g_n(\nu_m) + t_m g_n(\mu_m)\big]\Big) \\
&= (1-t_m)\varphi_{mn-2}([g_n(\nu_m)])+t_m g_n(\mu_m) \\
&= (1-t_m) \varphi_{mn-2}(\overline{g}_{m-2,n}([\nu_m])) + t_m g_n(\mu_m)\\
&= (1-t_m) \varphi_{(m-2)n}(\overline{g}_{m-2,n}([\nu_m])) + t_m g_n(\mu_m),
\end{align*}
where for the final step, we observe that because $\nu_m$ has at most $m-2$ arcs, $g_n(\nu_m)$ has at most $(m-2)n$ arcs, so applying~\eqref{eq: map phi} repeatedly shows $\varphi_{mn-2}(\overline{g}_{m-2,n}([\nu_m])) = \varphi_{(m-2)n}(\overline{g}_{m-2,n}([\nu_m]))$.
Therefore $\varphi_{mn}(\overline{g}_{m,n}([\mu])) = h_{m,n}(\varphi_m([\mu]))$, so we have shown that the diagram commutes.
This finishes the inductive step and completes the proof of Theorem~\ref{thm:S^1-equivalence}. 
\end{proof}

\FloatBarrier

\section{Conclusion}
\label{sec:conclusion}

Whenever we have a filtration of $G$-spaces $Y_1 \hookrightarrow Y_2 \hookrightarrow \ldots \hookrightarrow Y_{k-1} \hookrightarrow Y_k$,
we can define the associated persistent equivariant cohomology module
\[H^*_G(Y_k) \to H^*_G(Y_{k-1}) \to \ldots \to H^*_G(Y_2) \to H^*_G(Y_1).\]
If $Y_k$ is $G$-contractible, 
then $H^*_G(Y_k)=H^*_G(\pt)=H^*(BG)$, and so we can view persistent equivariant cohomology as a way to combine all of the maps $Y_i \to \pt$ (showing that the equivariant cohomology $H^*_G(Y_i)$ is a module over the ring $H^*(BG)$) into a unified sequence of maps 
\[H^*(BG) \to H^*_G(Y_{k-1}) \to \ldots \to H^*_G(Y_2) \to H^*_G(Y_1).\]

In this paper, we focused on the setting where $X$ is a metric space equipped with an action of $G$ by isometries, and hence we have an induced action of $G$ on the Vietoris--Rips metric thickening $\vrm{X}{r}$ for all $r\ge 0$.
One could instead replace the Vietoris--Rips thickening with the intrinsic \v{C}ech metric thickening~\cite{AAF} defined via
\[\cechm{X}{r}\coloneqq \left\{\sum_{j=0}^n a_j \delta_{x_j}~\Bigg|~n\ge 0,\ a_j\ge 0,\ \sum_j a_j=1,\ \cap_{j=0}^n B(x_j,r)\neq \emptyset\right\},\]
where $B(x_j;r)$ is the ball in $X$ of radius $r$ about the center point $x_j$.
One could also study sublevelset persistent equivariant cohomology: given a $G$-equivariant space $Y$ and a $G$-invariant function $f\colon Y\to \R$, we get a filtration of $G$-equivariant sublevelsets
\[ f^{-1}((-\infty,r_1]) \hookrightarrow f^{-1}((-\infty,r_2]) \hookrightarrow \ldots \hookrightarrow f^{-1}((-\infty,r_{k-1}]) \hookrightarrow f^{-1}((-\infty,r_k]) \]
for any $r_1 \le r_2 \le \ldots \le r_{k-1}\le r_k$.
If $f$ happens to be a Morse function then there is a close connection to Morse--Bott theory (see Question~\ref{ques:Morse-Bott}), but one can still define the persistent equivariant cohomology of this sublevelset filtration even if $f$ is not Morse.

We end by advertising a list of open questions, which we hope will inspire future collaborations between mathematicians with training in applied topology and mathematicians with training in (equivariant) homotopy theory.

\begin{question}
What is the persistent $S^1$-equivariant cohomology of the intrinsic \v{C}ech metric thickienings $\cechm{S^1}{r}$ and complexes $\cech{S^1}{r}$ of the circle?
The intrinsic \v{C}ech simplicial complex $\cech{S^1}{r}$ of the circle is the nerve of balls drawn in the circle (i.e.\ the nerve of circular arcs), not the nerve of balls in $\R^2$.
These \v{C}ech complexes $\cech{S^1}{r}$ also obtain the homotopy types of all odd-dimensional spheres~\cite[Section~9]{AA-VRS1}, although at different scale parameters and with different fixed point sets than the Vietoris--Rips thickenings of the circle.
\end{question}

\begin{question}
What is the persistent $S^1$-equivariant cohomology of the anti-Vietoris--Rips metric thickenings of the circle?
See~\cite[Chapter~5]{MoyThesis} and~\cite{AEMM}.
\end{question}

\begin{question}
If $X_k$ is a set of $k$ evenly-spaced points on the circle, then the Vietoris--Rips  and \v{C}ech complexes $\vr{X_k}{r}$ and $\cech{X_k}{r}$ are equipped with an action by the dihedral group $D_{2k}$.
What are the persistent $D_{2k}$-equivariant cohomologies of $\vr{X_k}{r}$ and $\cech{X_k}{r}$?
We note that the homotopy types of $\vr{X_k}{r}$ and $\cech{X_k}{r}$ are always either a single odd-dimensional sphere or a wedge sum of even-dimensional spheres of the same dimension~\cite{Adamaszek2013,AAFPP-J}.
We refer the reader to~\cite[Section~7.2]{AAFPP-J} for how the dihedral group acts on the homology of $\cech{X_k}{r}$.
\end{question}

\begin{question}
\label{ques:S^n-pec}
Let $\so(n+1)$ be the special orthogonal group of all rotations of $\R^{n+1}$.
Since $\so(n+1)$ acts by isometries on the $n$-sphere $S^n$, we get an induced action of $\so(n+1)$ on the Vietoris--Rips and intrinsic \v{C}ech metric thickenings $\vrm{S^n}{r}$ and $\cech{S^n}{r}$.
What are the persistent $\so(n+1)$-equivariant cohomologies of the Vietoris--Rips and \v{C}ech thickenings of the $n$-sphere?
\end{question}

\begin{question}
The Vietoris--Rips metric thickening $\vrm{S^n}{r}$ of the $n$-sphere is homotopy equivalent to $S^n$ for all $r<r_n$, and furthermore $\vrm{S^n}{r_n} \simeq S^n * \frac{\so(n+1)}{A_{n+2}}$~\cite[Theorem~5.4]{AAF}.
Here, $r_n\coloneqq \arccos(\frac{-1}{n+1})$ is the diameter of the vertex set of a regular $(n+1)$-simplex inscribed in $S^n$, and the alternating group $A_{n+2}$ encodes the rotational symmetries of a regular $(n+1)$-simplex inscribed in $S^n$.
What are the homotopy types of $\vrm{S^n}{r}$ for $r_n<r<\pi$, and could information about Question~\ref{ques:S^n-pec} be helpful for determining these homotopy types?
\end{question}

\begin{question}
\label{ques:S^n-cech}
The intrinsic \v{C}ech metric thickening $\vrm{S^n}{r}$ of the $n$-sphere is homotopy equivalent to $S^n$ for sufficiently small scales; this is a metric analogue (see~\cite[Theorem~4.4]{AAF} and~\cite[Theorem~5.5]{AM}) of the nerve lemma (\cite[Corollary~4G.3]{Hatcher}).
But for $r>\frac{\pi}{2}$, analogues of the nerve lemma will no longer apply, since the cover of $S^n$ by balls of radius $r>\frac{\pi}{2}$ is no longer a good cover.
What are the homotopy types of $\cech{S^n}{r}$ for $\frac{\pi}{2}<r<\pi$, and could information about Question~\ref{ques:S^n-pec} be helpful for determining these homotopy types?
\end{question}

\begin{question}
Since $S^3$ is a Lie group, we have an action of $S^3$ on itself, which induces an action of $S^3$ on the Vietoris--Rips and \v{C}ech metric thickenings $\vrm{S^3}{r}$ and $\cechm{S^3}{r}$ as $r$ increases.
What is the persistent $S^3$-equivariant cohomology of $\vrm{S^3}{r}$ and $\cechm{S^3}{r}$?
For other choices of Lie groups $M$, what is the persistent $M$-equivariant cohomology of $\vrm{M}{r}$ and $\cechm{M}{r}$?
\end{question}

\begin{question}
What is the persistent equivariant cohomology of the action of the hyperoctahedral group on $\vr{Q^n}{r}$, where $Q_n$ is the vertex set of the $n$-dimensional hypercube, equipped with the Hamming metric?
See~\cite{adamaszek2022vietoris,adams2024lower,bendersky2023connectivity,briggs2024facets,carlsson2020persistent,feng2023homotopy,feng2023vietoris,saleh2024vietoris,shukla2022vietoris} for information on the homotopy types and homology groups of these complexes.
\end{question}

\begin{question}
\label{ques:Morse-Bott}
One can compute the persistent cohomology of any filtration on a space, whether or not that filtration comes from a Morse function.
But, if the filtration does come from a Morse function, then there is a close connection with Morse theory: every starting value or ending value of a bar in the persistent cohomology module corresponds to the value of some (isolated) critical point of the Morse function.
Similarly, one can compute the persistent equivariant cohomology of any equivariant filtration on an equivariant space.
But, if the filtration comes from an equivariant Morse--Bott function, then there should be a close connection with Morse--Bott theory: every starting or ending value of a bar in the persistent equivariant cohomology should correspond to the value of some critical submanifold of the Morse--Bott function.
What is the precise relationship between persistent equivariant cohomology and Morse--Bott functions?
We refer the reader to~\cite{austin1995morse} for a description of the relationship between equivariant cohomology and Morse--Bott theory.
\end{question}

\begin{question}
The following question is due to Matthew Zaremsky.
Which infinite groups admit a contractible Vietoris--Rips complex?
Does every group that acts properly cocompactly on a contractible complex admit a contractible Vietoris--Rips complex (with respect to a word metric from some finite generating set)?
We refer the interested reader to~\cite{zaremsky2022bestvina,Zaremsky-VRtalk2021} for a phrasing of this question as when a group of ``type $F_*$'' can be revealed as such via Vietoris--Rips complexes.
See~\cite{373007} for remarks on integer lattices with standard and non-standard word metrics, and for results on right-angled Artin groups.
With the word metric using the standard generators $e_1=(1,0,\ldots,0)$, $e_2=(0,1,0,\ldots,0)$, \ldots, $e_n=(0,\ldots,0,1)$, Virk recently showed that $\vr{\Z^n}{r}$ is contractible for all $r$ sufficiently large~\cite{virk2024contractibility}, although it is not yet proven if $r\ge n$ suffices.
\end{question}

\section{Acknowledgment}

We would like to thank Tom Goodwillie for his helpful answers to our Math Overflow question~\cite{457002}, which in particular pointed us to the Gysin sequence.

\bibliographystyle{plain}
\bibliography{PersistentEquivariantCohomology.bib}

\end{document}